%% file: sosCertEVIs.tex
\documentclass[conference, onecolumn,11pt]{IEEEtran}

\IEEEoverridecommandlockouts

\let\labelindent\relax

\usepackage{amsmath,amsthm,amssymb}
\usepackage{subcaption}
\usepackage{tikz}
\usetikzlibrary{fadings}
\tikzfading[name=fade out,
            inner color=transparent!20,
            outer color=transparent!90]
\tikzfading[name=fade equal,
            inner color=transparent!30,
            outer color=transparent!30]
\usetikzlibrary{circuits.ee.IEC}
\usepackage[american]{circuitikz}
\usetikzlibrary{calc}
\usepackage{color}

\usepackage[obeyFinal]{todonotes}
\usepackage{mathtools, amssymb}
\allowdisplaybreaks
\usepackage{url}

\usepackage[boxruled]{algorithm2e}
  \renewcommand\AlCapFnt{\normalfont\small}
  \renewcommand\AlCapNameFnt{\normalfont\small}
\usepackage[font=footnotesize]{caption}

\usepackage{enumitem}%[inline]
\setlist[enumerate, 1]{topsep=-0.5ex}
\setlist[itemize, 1]{topsep=-0.5ex}

\newlist{assump}{enumerate}{1}
\setlist[assump]{label = {\bf (A\arabic*)}, resume}
\newlist{hyp}{enumerate}{1}
\setlist[hyp]{label = {\bf (H\arabic*)}, resume}
\newlist{lyap}{enumerate}{1}
\setlist[lyap]{label = {\bf (L\arabic*)}, resume}
\newlist{state}{enumerate}{1}
\setlist[state]{label = {\bf (S\arabic*)}, resume}
\newlist{myprob}{enumerate}{1}
\setlist[myprob]{label = {\bf Problem~\arabic*.}, resume}

\theoremstyle{plain}
\newtheorem{thm}{Theorem}
\newtheorem{prop}[thm]{Proposition}
\newtheorem{lem}[thm]{Lemma}

\theoremstyle{remark}
\newtheorem{rem}[thm]{Remark}
\theoremstyle{definition}
\newtheorem{defn}{Definition}
\newtheorem{exam}{Example}
\newtheorem{assmpt}{Assumption}

\renewcommand{\geq}{\geqslant}
\renewcommand{\ge}{\geqslant}
\renewcommand{\leq}{\leqslant}
\renewcommand{\le}{\leqslant}

\newcommand{\B}{\mathbb{B}}

\newcommand{\N}{\mathbb{N}^\ast}
\newcommand{\Nz}{\mathbb{N}}

\newcommand{\R}{\mathbb{R}}

\newcommand{\cA}{\mathcal{A}}

\newcommand{\cC}{\mathcal{C}}

\newcommand{\cI}{\mathcal{I}}

\newcommand{\cK}{\mathcal{K}}

\newcommand{\cN}{\mathcal{N}}

\newcommand{\cP}{\mathcal{P}}

\newcommand{\cR}{\mathcal{R}}

\newcommand{\cT}{\mathcal{T}}
\newcommand{\cU}{\mathcal{U}}

\newcommand{\argmin}{\operatornamewithlimits{arg\,min}}

\DeclareMathOperator{\dom}{dom}

\DeclareMathOperator{\inn}{int}
\DeclareMathOperator{\bd}{bd}

\DeclareMathOperator{\proj}{proj}

\DeclareMathOperator{\cl}{cl}

\DeclareMathOperator{\lcp}{LCP}
\DeclareMathOperator{\lccp}{LCCP}

\newcommand{\eps}{\varepsilon}

\newcommand{\wt}{\widetilde}
\newcommand{\wh}{\widehat}

\colorlet{Darkred}{red!50!black}
\colorlet{Darkgreen}{green!50!black}
\xdefinecolor{UIUCblue}{rgb}{0,0.24,0.49} %professional:{110,139,191}
\xdefinecolor{UIUCorange}{rgb}{0.96,0.5,0.14} %professional:{239,138,28}

\usepackage{color}

\title%[SOS Methods for Stability of Nonsmooth Systems]
{\huge Cone-Copositive Lyapunov Functions for Complementarity Systems: Converse Result and Polynomial Approximation}
\author{Marianne Souaiby \and Aneel Tanwani \and Didier Henrion
\thanks{The authors are with LAAS-CNRS, Universit\'e de Toulouse, 31400 Toulouse, France. D. Henrion is also with the Faculty of Electrical Engineering of the Czech Technical University in Prague, Czechia. This work is sponsored by the ANR project {\sc ConVan} with grant number ANR-17-CE40-0019-01.}
\thanks{This manuscript is a preprint of the article accepted for publication in IEEE Transactions on Automatic Control, and is available on publisher's website via IEEE Early Access. It is scheduled to appear in printed version in March 2022.}
}

\begin{document}
\maketitle

\begin{abstract}
This article establishes the existence of Lyapunov functions for analyzing the stability of a class of state-constrained systems, and it describes algorithms for their numerical computation. The system model consists of a differential equation coupled with a set-valued relation which introduces discontinuities in the vector field at the boundaries of the constraint set. In particular, the set-valued relation is described by the subdifferential of the indicator function of a closed convex cone, which results in a cone-complementarity system. The question of analyzing stability of such systems is addressed by constructing cone-copositive Lyapunov functions. As a first analytical result, we show that exponentially stable complementarity systems always admit a continuously differentiable cone-copositive Lyapunov function. Putting some more structure on the system vector field, such as homogeneity, we can show that the aforementioned functions can be approximated by a rational function of cone-copositive homogeneous polynomials. This later class of functions is seen to be particularly amenable for numerical computation as we provide two types of algorithms for precisely that purpose. These algorithms consist of a hierarchy of either linear or semidefinite optimization problems for computing the desired cone-copositive Lyapunov function. Some examples are given to illustrate our approach.
\end{abstract}


\input{introduction}

\input{prelims}

\input{converseLip}
%\input{converseConv}

\input{homoPoly}

\input{sosAlgos}

\input{examples}

\section{Conclusions}\label{sec:conc}

This article addressed the stability analysis for a class of complementarity systems using the method of Lyapunov functions. Questions pertaining to the existence of cone-copositive Lyapunov functions were answered in the affirmative for exponentially stable systems. Some refinements of this result, under certain conditions on the vector field in the system dynamics, allow us to restrict our search for cone-copositive Lyapunov functions within the class of homogeneous and rational polynomials. These statements indeed bring tractability to the numerical methods that have been proposed in this paper for computing Lyapunov functions. In particular, two hierarchies of convex optimization problems are obtained using the methods based on discretization and SOS approximation, respectively.

Several immediate questions of interest emerge from our work which require further investigation. The first one among those is to extend our results to broader classes of complementarity systems. Systems of the form~\eqref{eq:compSys} are one particular class of relative degree one systems, but in applications, one sees more complex complementarity systems of the form studied in \cite{TanwBrog18}. In such a wider class of systems, one sees different kinds of constraints on the state trajectories. Moreover, the constraints may vary with time in which case one has to consider the possibility of time-varying Lyapunov functions. It would be interesting to consider converse questions for this broader class of systems.

Some extensions at the level of designing algorithms are also of potential interest. At this moment our algorithms are specifically adapted to the constraint sets of the form positive orthant or an invertible linear transformation of such sets. Adapting this technique to more generic sets remains to be seen. Also, in our current treatment, we have considered discretization algorithms in $\R^2$ and $\R^3$, where it is relatively straightforward to write algorithms for partition of simplices. It remains to be seen how the algorithms for simplical partition in higher dimensions perform in computing Lyapunov functions. 

%As a possible direction for future research, we can consider further generalizations of the complementarity constraints in \eqref{eq:compSys}. This is done by replacing the conic structure and allow for possibly bounded polyhedra \cite{TanwBrog18}, which is captured by affine complementarity conditions:
%\[
%0 \le Fx + q \perp \eta \ge 0
%\]
%where the matrix $F$, and the vector $q$ define the polyhedron $\{x \in \R^n \, \vert \, Fx + q \ge 0\}$ to which the state is constrained to evolve.

%\begin{itemize}
%    \item More focus is needed for developing converse results, and theoretically determine the class of functions for making precise our search of Lyapunov functions.
%    \item In terms of analysis, more rigorous results are required for studying the copositivity of the polynomials.
%\end{itemize}

\section*{Acknowledgements}
The authors would like to thank Bachir El Khadir for pointing out the reference \cite{AhmKha19} during his visit to Toulouse. The second author also acknowledges several discussions with his PhD student Matteo Della Rossa related to the existing results on converse Lyapunov theorems.

\input{appendixA}

\bibliographystyle{plain}

\end{document}

%% file: introduction.tex
% !TEX root = sosCertEVIs.tex
%%%%%%%%%%%%%%%%%%%%%%%%%%%%%%%%%%%

\section{Introduction}\label{sec:Intro}
Lyapunov functions provide a useful tool for the stability analysis of dynamical systems. Several advances have been made on the theoretical side to establish existence of Lyapunov functions for various classes of dynamical systems, see e.g. \cite{Khalil02, Libe03, GoebSanf12} for examples of standard expositions. The fundamental question in most of these works boils down to checking the positivity of certain functions over the state space, which is a challenging problem numerically \cite{MurtKaba87}. Modern developments in the field of real algebraic geometry \cite{Puti93, Schm91} provide certificates of positivity of (polynomial) functions with Positivstellens\"atze relying on {\em sums-of-squares} (SOS) decompositions. Since it has been observed in \cite{Powers98} that finding SOS decompositions is equivalent to semidefinite programming (SDP) or linear matrix inequalities (LMI), numerical tools based on SOS optimization have been developed extensively over the past two decades to compute Lyapunov functions, see e.g. \cite{Parr00, PrajPapa02, HenrGaru05, Chesi09}.

Stability analysis of hybrid, or nonsmooth dynamical systems, where the vector field is set-valued with possible discontinuities, is of particular relevance with respect to several applications.  Naturally, Lyapunov functions for such systems provide a potent tool for studying stability related properties as well. 
%
%For what follows in this paper, it is helpful to recall the papers \cite{MasoBosc06, Peet09}, where the authors prove existence of {\em polynomial} Lyapunov functions for switched linear systems, and nonswitched nonlinear systems. For this class of polynomial functions, if the search space is further restricted to homogeneous functions, then the computation of Lyapunov functions becomes extremely amenable to algorithms using SOS decomposition to solve polynomial inequalities \cite{AhmPar17}. However, most of these converse results that establish the existence of Lyapunov functions require the system to be exponentially stable. Unsurprisingly, there exist examples of asymptotically stable nonlinear systems which do not admit a polynomial Lyapunov function \cite{AhmKrsPar11}, and hence a care must be taken when applying SOS-based tools for computation of Lyapunov functions.
%
%\old{This paper also addresses the application of semidefinite programming, or sums-of-squares techniques to computing Lyapunov functions for a certain class of nonsmooth and hybrid dynamical systems.}
When the system is modeled by switching vector fields over the whole state space, then the construction of Lyapunov functions using SOS is studied in \cite{PapaPraj09, AhmPar17, AhmaJung18}. However, we are concerned with a certain class of differential inclusions which is useful in modeling systems with state constraints, where the vector field exhibits discontinuous behaviour on the boundary of the constraints so that the state trajectory is forced to evolve within the prespecified set. In the literature, there are several frameworks for modeling this behaviour, such as {sweeping processes}, {projected dynamical systems} or {\em complementarity systems} \cite{Brog03}, \cite{BrogTanw20}, \cite{CamlPang06}. The relevance of these systems is seen in many practical systems encountered in engineering, physics and biology. For example, in mechanics, the interaction between multiple rigid bodies or between a rigid body and the environment can be modelled using nonsmooth force laws for contact, impact and friction. Some variants of these systems are also studied in \cite{TanwBrog14, TanwBrog18} in the context of control-theoretic problems.

The stability analysis of complementarity systems using Lyapunov functions has received some attention in the literature. Since the state of such systems essentially evolves in a closed convex cone, often chosen to be the positive orthant, it is naturally desirable to consider Lyapunov functions which are positive definite over the positive orthant; the functions satisfying this latter property are called {\em copositive} functions. The need to search such functions for stability analysis of complementarity systems was presented as an open problem in \cite{CamlSchu04}. The papers \cite{GoelMotr03, GoelBrog04, CamlPang06} investigate sufficient stability conditions for linear complementarity systems, or conewise linear systems \cite{IanIerVas19} in terms of copositive Lyapunov functions. The paper \cite{GoelBrog04} also provides examples of systems where a positive definite Lyapunov function does not exist, but the system is nonetheless asymptotically stable and it admits a copositive Lyapunov function.

While these existing works have shown the utility of enlarging the search space of Lyapunov functions from positive definite to copositive functions, and cone-copositive functions when considering systems with state trajectories constrained to a cone rather than the positive orthant, none of the existing works has addressed the converse question:
\begin{center}
Does every asymptotically stable complementarity system admit a cone-copositive Lyapunov function?
\end{center}
The first objective of this paper is to answer this question in the affirmative by constructing a Lyapunov function as a functional of the solution trajectories, thereby concluding that one does not need to go beyond cone-copositive functions to find Lyapunov functions for complementarity systems. By putting more structure on the system dynamics, and using the appropriate density results, we are able to prove the existence of a cone-copositive Lyapunov function which can be expressed as a ratio of homogeneous polynomials. Converse stability results for dynamical systems have been studied for a long time in control community, see the recent survey article \cite{Kell15}. Moreover, due to discontinuities in the vector field at the boundary of the constraint set (which can be seen as an example of constrained switching), establishing the existence of Lyapunov functions within cone-copositive functions becomes difficult.

The second objective of this paper is to propose computationally tractable algorithms for finding the Lyapunov functions. The interesting aspect of our problem lies in computing Lyapunov functions which satisfy certain inequalities over a given set. For example, in linear complementarity systems, one needs to check the positivity of a function over the positive orthant only, and if the function we seek is of the form $x^\top P x$, then finding such a function boils down to finding a {\em copositive} matrix $P$ that satisfies certain inequalities. However, checking whether a given matrix is copositive is an NP-hard problem \cite{BomzDur00}. The papers \cite{BunDur08, BunDur09, NiYaZha18}, \cite{Dur10} propose algorithms for detecting copositivity of a matrix or tensor. Moreover, we will show with the help of an example that, even in the case of linear complementarity systems, such functions cannot be computed by solving a linear set of equations, as is done for unconstrained linear systems. Another challenging aspect of these problems is that, when dealing with conic constraints which are unbounded sets, there are no readily available {Positivstellensatz} that guarantee SOS decompositions of a positive polynomial over the sets of our interest. The field of copositive programming has been active area of research over the past decade which addresses some of these challenges \cite{Bomz12}. In computing the Lyapunov functions for complementarity systems which evolve on unbounded cones with positivity constraints, we are faced with similar challenges. 

Motivated by such questions, we propose two approaches for computing homogeneous cone-copositive Lyapunov functions numerically. The first one is a {\em discretization method} which is based on finding an inner approximation of the cone of cone-copositive polynomials by using simplicial partitions and evaluating inequalities over a set of points taken on the simplex. It is shown that, as the partition gets finer, we can approximate any cone-copositive polynomial function. The second approach is an {\em SOS method} where we show that the positivity of polynomial over the given cone can be checked by expressing it as an SOS function. By increasing the degree of the approximating SOS polynomial, we again obtain a hierarchy of SDP problems to compute the desired Lyapunov function. Then, we derive the corresponding algorithms for those two techniques, which can be seen as an adaptation of tools available in the literature on polynomial optimization. The illustration of some academic examples is provided using standard Matlab toolboxes.

%The paper is organized as follows. In section~\ref{sec:prelim} we give some preliminaries from convex analysis and define the system class of our study. In section~\ref{ProblemForm} we present the stability notions, motivate our work by some examples and describe the two problems studied in this paper. In Section~\ref{sec:ConvCopLya}, we prove the existence of {\color {red} cone-}copositive Lyapunov functions for complementarity systems. Section~\ref{HomPolLyap} describes the existence of homogeneous and polynomial Lyapunov approximation. In section~\ref{sec:comp}, we provide our two approaches with their algorithmic implementations, the discretization method and the SOS method, which will let us calculate the approximation of such functions numerically. Some examples and simulations illustrate those results in Section~\ref{sec:simul}, followed by a conclusion in Section~\ref{sec:conc}.

%% file: prelims.tex
% !TEX root = sosCertEVIs.tex
%%%%%%%%%%%%%%%%%%%%%%%%%%%%%%%%%%%

\section{System Class}\label{sec:prelim}
We begin this section by introducing some basic notions from convex analysis which will be used for describing the class of dynamical systems studied in this paper.
\subsection{Dynamical System with Constrained Trajectories}
We are interested in studying a class of dynamical systems described by the variational inequalities
\begin{equation}\label{diffnormalincl}
    \dot x(t) \in f(x(t)) - \cN_{S}(x(t)), \enspace \text{a.e.} \enspace t \ge 0,
\end{equation}
where $f:\R^n \to \R^n$ is a given vector field, $x(t) \in \R^n$ denotes the state, $S$ is a given closed convex subset of $\R^n$ containing the origin.
The normal cone to $S$ at $x$ is defined by
\begin{equation}\label{eq:defnormalcone}
    \cN _{S}(x):=\left\{\lambda \in \R^n \:\vert\: \left\langle \lambda, x'-x\right\rangle \le 0, \forall x' \in S \right\}.
\end{equation}
If $x \in \inn(S)$, the interior of $S$, then $\cN _{S}(x)={0}$ and by convention, we let $\cN _{S}(x):= \emptyset $ for all $x \not\in S$.

The differential inclusion \eqref{diffnormalincl} is a particular case of the following class of variational inequalities
\begin{equation}\label{diffincl}
\dot x(t) \in f(x(t)) - \partial \varphi (x(t)), \enspace \text{a.e.} \enspace t \ge 0
\end{equation}
where $\varphi : \R^n \to \R \cup \left\{ +\infty \right\}$ is a proper, convex and lower semicontinuous function, and the subdifferential of $\varphi$ at $x \in \R^n$ is defined as
%\begin{equation}\label{eq:defSubdiff}
$\partial \varphi (x) := \left\{ \lambda \in \R^n \:\vert\: \left\langle \lambda, z-x \right\rangle \le \varphi(z)-\varphi(x), \forall z \in \dom (\varphi) \right\}$, 
%\end{equation}
with $\dom (\varphi) := \left\{x\in \R^n \:\vert\: \varphi(x) < +\infty \right\}$.
Indeed, if one denotes the indicator function of a closed convex set $S \subset \R^n$ by $\psi_S(\cdot)$, that is, $\psi_S(x) = 0$, if $x \in S$ and $\psi_S(x) = +\infty$ if $x \not\in S$, then~\eqref{diffnormalincl} is obtained from \eqref{diffincl} by choosing $\varphi=\psi_{S}$. Inclusion~\eqref{diffnormalincl} captures the class of complementarity systems studied in this paper, but the framework of \eqref{diffincl} is necessary for a broader class of complementarity systems such as the ones studied in \cite{CamlPang06, TanwBrog18}.

The formalism of system \eqref{diffnormalincl} with inclusion naturally allows us to describe dynamics constrained to evolve in set $S$. Using the depiction in Figure~\ref{fig:setMap}, it is seen that, during the evolution of a trajectory, if $x(t)$ is in interior of $S$, then $\cN_{S}(x(t)) = 0$ and the motion of the trajectory continues according to the differential equation $\dot x(t) = f(x(t))$. While $x(t)$ is at the boundary, we add a vector from the set $-\cN_{S}(x(t))$, which restricts the motion of the state trajectory in tangential direction on the boundary of the constraint set $S$.

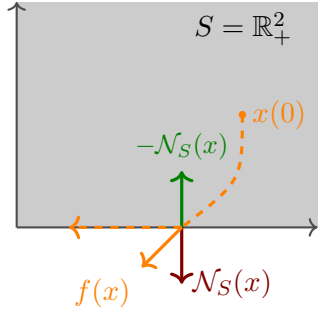
\begin{figure}
\centering
\begin{tikzpicture}
\def\myLength{72}
\def\myWidth{18}
\def\myAngle{30}
\def\horiz{0.867*\myLength}
\def\vert{0*\myLength}
\coordinate (start) at (3,1.5);
\coordinate (origin) at (0,0);
\fill [{black!20}] (0,0)--(0,3)--(4,3)--(4,0)--(0,0);
\draw [thick, black!70,->] (0,0) -- (0,3);
\draw [thick, black!70,->] (0,0) -- (4,0);
\draw (4,3) node[anchor=north east, text width=1.5cm,align=left] {$S = \R^2_+$};
\draw [Darkred, very thick, ->] (\horiz pt, \vert pt) -- +(270: 0.75cm) node[anchor=west] {$\cN_{S}(x)$};
\draw [orange, very thick, ->] (\horiz pt, \vert pt) -- +(225: 0.75cm) node[anchor=north east] {$f(x)$};
\draw [Darkgreen, very thick, ->] (\horiz pt, \vert pt) -- +(90: 0.75cm) node[anchor=south] {\small $-\cN_{S}(x)$};
\draw [orange, very thick,dashed, ->] (\horiz pt, \vert pt) -- +(180: 1.5cm) node[anchor=west]{};
\draw (start) node[orange, circle,fill,inner sep=1] {} node[anchor=west, orange] {$x(0)$};
\draw [orange, rounded corners,very thick,dashed] (start) .. controls +(down:25pt) .. (\horiz pt,\vert pt) ;
\end{tikzpicture}
\caption{State trajectories in constrained system with $S = \R_+^n$.}
\vskip -2em
\label{fig:setMap}
\end{figure}

\subsection{Complementarity Systems}

In this article, we focus on the particular class of constrained systems where the admissible set $S$ is a cone, denoted by $K$. Hence for each $x \in K$, we have $\lambda x \in K$ for each $\lambda \in \R_{\ge 0}$, and for all $\alpha,\beta \in \R_{\ge 0}$ and $x,y \in K$, we have $\alpha x + \beta y \in K$. The dual cone $K^\star \subset \R^n$ of a cone $K \subset \R^n$ is defined as
\begin{equation}\label{set:dualcone}
    K^\star:=\left\{p \in \R^n \:\vert\: \left\langle p,v \right\rangle \ge 0, \forall \, v \in K \right\}.
\end{equation}
% To draw connections with the system class~\eqref{diffincl}, 
We recall a basic result from convex analysis \cite[Proposition~1.1.3]{FaccPang03}:
\[\eta \in -\cN_K(x) \Longleftrightarrow K^\star \ni \eta \perp x \in K\]
where the notation $K^\star \ni \eta \perp x \in K$ is the short-hand for three statements: i) $x \in K$, ii) $\eta \in K^\star$, and iii) $x^\top \eta = 0$.
With these basic definitions, we introduce the following class of systems for which we develop the main results of this paper.
\begin{defn}[Complementarity System]
Given a function $f : \R^n \to \R^n$ and a cone $K \subset \R^n$, a complementarity system is described by the following differential equation:
\begin{equation}
\tag{C-Sys}
\begin{array}{c}
\dot x = f(x) + \eta \\ %\label{eq:compSysa}\\
K^\star \ni \eta \perp x \in K. %\label{eq:compSysb}
\end{array}
\label{eq:compSys}
\end{equation}
\end{defn}
Several works exist in the literature which deal with existence and numerical construction of the solution to system \eqref{eq:compSys}. A recent reference \cite{CamlIann19} contains results in this direction, along with pointers to earlier works.
Motivated by these works, it is stipulated in the remainder of this paper that the data of \eqref{eq:compSys} satisfy the following assumption.
\begin{assmpt}\label{lip}
Function  $f:\R^n \to \R^n$ is locally Lipschitz continuous, $f(0)=0$ and $K \subset \R^n$ is a closed convex cone.
\end{assmpt}

\subsection{Solution of Complementarity Systems}

Before proceeding with the problem formulation and the corresponding results, it is instructive to recall how a solution to \eqref{eq:compSys} evolves with time, and the underlying optimization problem which may be solved to compute $\eta$.
For a fixed $s \ge 0$, if $x(s) \in \inn (K)$, then $\cN_{K}(x(s)) = \{0\}$, and we let $\eta(s) = 0$. As a result, for some $\eps>0$ and $t \in [s,s+\eps)$, we have $\dot x(t) = f(x(t))$ and $x(t) \in \inn (K)$. However, if for some $\bar s \ge 0$, we have that $x (\bar s) \in \bd(K)$, the boundary of $K$, then we essentially compute $\eta(\bar s)$ satisfying the relation\footnote{Note that, for a closed convex cone $K \subseteq \R^n$, for each $x \in K$, we denote the tangent cone to $K$ at $x$ by $\cT_K(x)$ and $\cN_K(x) = -\cT_{K}(x)^\star$. Also, for $x \in K$, if $\eta \in -\cN_K(x)$, then $\eta \in K^\star$, and hence $K^\star \subseteq \cT_K(x)^\star$.}
\begin{equation}\label{eq:LCCPTangent}
K^\star \ni \eta (\bar s) \perp f(x(\bar s)) + \eta(\bar s) \in \cT_{K}(x).
\end{equation}
If the boundary constraint remains active over an interval $[\bar s, \bar s+\eps]$ for some $\eps > 0$, that is, for each $t \in [\bar s, \bar s+\eps]$, $x(t) \in \bd(K)$, then $\eta (t)$ satisfies the complementarity relation in \eqref{eq:LCCPTangent}. Using the notation in Appendix~\ref{app:basicsComp}, we say that $\eta(t) \in \lccp(f(x(t)),I,\cT_K (x(t)))$, which is equivalently described as the solution to the optimization problem stated in \eqref{optimprob}.

In what follows, it is also important to recall how we interpret the solution to \eqref{eq:compSys} if $x(0) = x_0 \not\in K$. In such a case, we let
\begin{equation}
x_0^+ = \proj_K(x_0) := \argmin_{z\in K} \| x_0 - z \|
\end{equation}
and then propagate the solution with $x_0^+$, the projection of $x_0$ on $K$ with respect to Euclidean norm. We can thus formally define the solution to \eqref{eq:compSys} as follows:

\begin{defn}\label{def:solBasic}
For a given initial condition $x_0 \in \R^n$ and an interval $[0,T]$, a solution to \eqref{eq:compSys} is an absolutely continuous function $x:[0,T] \to \R^n$, such that $x(t) \in K$ for each $t > 0$, and $x_0^+ = \proj_K(x_0) := \argmin_{z\in K} \vert x_0 - z \vert$.
\end{defn}

Under Assumption~\ref{lip}, there exists a unique solution to \eqref{eq:compSys} in the sense of Definition~\ref{def:solBasic}.
We denote by $x(t;x_0)$ the solution of \eqref{eq:compSys}, at time $t \ge 0$ starting with initial condition $x_0$ at time $0$.
Assumption \ref{lip} also guarantees that the origin is an equilibrium and $x(t;0)=0$ is the unique trivial solution starting from $x_0 = 0$.
Indeed, with $K$ being a closed convex cone, we have $0 \in K$. Under the condition $f(0) = 0$, we have $\eta(t) = 0$ and $\dot x(t) = 0$, for all $t \ge 0$.

\section{Problem Formulation} \label{ProblemForm}

This article addresses some questions regarding the stability analysis of the trivial solution, the origin, for system \eqref{eq:compSys}, while assuming throughout that Assumption~\ref{lip} holds.
We first describe the appropriate notion of stability, and discuss some interesting properties that may arise due to the presence of constraints.

\subsection{Stability Notions}

We may now define as in \cite{GoelBrog04, GoelMotr03} the stability of the origin: it is stable if small perturbations of the initial condition at the origin lead to solutions remaining in the neighborhood of the origin for all forward times:
\begin{defn}[Stability]
The origin is stable in the sense of Lyapunov if for every $\eps > 0$ there exists $\delta>0$ such that 
\[
x_0 \in K, \|x_{0}\| \le \delta \Rightarrow \|x(t,x_0)\| \le \eps, \quad \forall t \ge 0.
\]
The origin is locally asymptotically stable if it is stable in the sense of Lyapunov and there exists $\beta >0$ such that 
\[
x_0 \in K, \|x_{0}\| \le \beta \Rightarrow \lim\limits_{t \rightarrow +\infty} \|x(t,x_0)\| =0.
\]
The origin is globally asymptotically stable if the latter implication holds for arbitrary $\beta > 0$. The origin is globally exponentially stable if there exists $c_0 > 0$ and $\alpha > 0$ such that $\|x(t,x_0)\| \le c_0 e^{-\alpha t} \| x_0\|$, for every $x_0 \in K$.
\end{defn}

Compared to the conventional definitions of stability for unconstrained dynamical systems, our domain of interest is reduced to the set $K$ in system \eqref{eq:compSys}. Also, the vector field jumps instantaneously at the boundaries of the set $K$, which may have an impact on the stability of the system. The following examples motivate why it is not enough to analyze stability just by looking at the vector field $f$ in \eqref{eq:compSys}.

\begin{exam}[Constraints make the system stable, even if the unconstrained system is unstable]\label{ex:consGAS}
Let $f(x) = Ax$ with $A = \begin{bsmallmatrix} -1 & -2\\ -1 &-1\end{bsmallmatrix}$, and $K = \R_+^2$. Matrix $A$ is not Hurwitz stable since one of its eigenvalues is in the right-half complex plane. However, constrained system \eqref{eq:compSys} is globally asymptotically stable, see our later Example \ref{ex:discret} in Section~\ref{sec:simul} for a proof based on a Lyapunov function.
\end{exam}

\begin{exam}[Constraints make the system unstable, even if the unconstrained system is stable]
Let $f(x) = Ax$ with $A = \begin{bsmallmatrix} -1.5 & -1\\ 2 &1\end{bsmallmatrix}$, and $K = \R_+^2$. The matrix $A$ is Hurwitz but the constrained system \eqref{eq:compSys} is unstable because on the $x_2$-axis, the vector field is pointing away from the origin.
\end{exam}

Note that in the interior of $K$, system \eqref{eq:compSys} follows the dynamics $\dot x = f(x)$. The first example, however, shows that even if the constrained system is globally asymptotically stable, it is not possible to work with a Lyapunov function for the unconstrained system. In Example~\ref{ex:consGAS}, the unconstrained system does not admit a positive definite function with negative definite time derivative over the entire state space. Consequently, one has to enlarge the search for Lyapunov functions to functions which are positive definite only on the admissible domain. The second example shows that even if one can find a Lyapunov function for the unconstrained system, it may not correspond to a Lyapunov function for the constrained system. Thus, the search of Lyapunov functions for the constrained system needs to be investigated differently from the unconstrained system.

\subsection{Cone-Copositive Lyapunov Functions}

Based on the above notions, one has to adapt the notion of Lyapunov functions when analyzing the stability of complementarity systems. It is thus of interest to introduce {\em cone-copositive} functions:

\begin{defn}[{Cone-copositivity and copositivity}]
Let $K \subseteq \R^n$ be a closed convex cone. A real-valued function $h:\R^n \to \R$ is said to be cone-copositive with respect to $K$, if $h(x) \ge 0$ for each $x \in K$. When $K = \R_+^n$, we simply say that $h$ is copositive.
\end{defn}

Positive definite functions are obviously cone-copositive, regardless of the cone under consideration. However, in general, when the cone $K$ is fixed, positive definite functions only form a subclass of the functions which are cone-copositive with respect to $K$.
With this function class, the following definition of Lyapunov functions for \eqref{eq:compSys} provides more flexibility:

\begin{defn}[{Cone-copositive Lyapunov Function}]\label{def:funcLyap}
System \eqref{eq:compSys} has a continuously differentiable (global) cone-copositive Lyapunov function $V:\R^n \to \R$ with respect to $K$ if
\begin{enumerate}
\item There exist class $\cK_\infty$ functions\footnote{A function $\alpha:\R_+ \to \R_+$ is said to be of class $\cK$ if it is continuous, it satisfies $\alpha(0) = 0$, and it is increasing everywhere on its domain. It is said to be of class $\cK_\infty$ if it is, in addition, unbounded.} $\underline \alpha$, $\overline \alpha$ such that
\[
 \underline \alpha(\| x \| ) \le V(x) \le \overline \alpha(\| x \|), \quad \forall \ x\in K;
 \]
\item There exists a class $\cK$ function $\alpha$ such that
\begin{subequations}\label{gradLyap}
\begin{gather}
\left\langle \nabla V(x),f(x) \right\rangle \le -\alpha(\| x \|), \quad \forall \, x \in \inn(K), \\
\left\langle \nabla V(x),f(x)+\eta \right\rangle \le -\alpha(\| x \|), \quad \forall x \in \bd(K), 
\end{gather}
\end{subequations}
where $\eta \in \lccp(f(x),I,\cT_K(x))$.
\end{enumerate}
Condition \eqref{gradLyap} is split into two parts because the complementarity variable resulting from a complementarity relation takes nonzero value only on the boundary of $K$.
\end{defn}

Note that we require the inequalities in \eqref{gradLyap} to hold \underline{only} for a particular selection of $\eta$. This aspect of our definition is in contrast with several existing works dealing with the existence of Lyapunov functions for differential inclusions \cite{ClarLedy98, TeelPral00}. Our first major question relates to the existence of Lyapunov functions in the sense of Definition~\ref{def:funcLyap}.%The search for cone-copositive functions is a hard problem in general.

{\bf Problem 1:} Does there exist a cone-copositive Lyapunov function for a stable complementarity system ?

We address Problem~1 in Section~\ref{sec:ConvCopLya} and our first main result in Theorem~\ref{thm:converseMain} provides conditions which guarantees existence of a cone-copositive Lyapunov function for exponentially stable systems. Building on this result, and imposing further assumptions on the vector field $f$ in \eqref{eq:compSys}, we are able to prove the existence of homogeneous Lyapunov functions, which are desired for computational reasons.

\subsection{Computations Using Numerical Approximations}

Our next target in the paper is to address the computational aspects of the Lyapunov functions for complementarity systems \eqref{eq:compSys}. While working with  homogeneous vector fields, we restrict our search to rational functions of homogeneous polynomials. It is observed that such functions are dense within the class of homogeneous differentiable functions, see \cite[Lemma~2.1]{AhmKha19}. Moreover, one can adapt the algorithms from the literature on copositive programming to compute these rational functions.

{\bf Problem 2:} If there exists a homogeneous rational cone-copositive Lyapunov function for a stable complementarity system, how can we construct it?

The answer to this question essentially boils down to finding certain polynomials which satisfy some nonnegativity condition. Such questions have again received a lot of attention in real algebraic geometry, and Positivstellens\"atze provide guidance in writing algorithms for the search of Lyapunov functions. We explore two possible routes.

The first method corresponds to a {\em discretization} in the set $K$ by taking points over a simplex. We evaluate the inequalities over that set of discrete points, and we solve for the desired coefficients of the polynomials. This discretization method provides an inner approximation of the cone of copositive polynomials, and it is seen that as the size of discretization step goes to zero, this inner approximation converges to the actual cone.

The second method relies on SOS decomposition of our function. Let $\R[x]$ denote the vector space of real polynomials in the variables $x=(x_{1},\dots,x_{n}) \in \R^n$. A multivariate polynomial $p(x)=p(x_{1},\dots,x_{n})$ is a sum of squares, abbreviated as $p$ is SOS, if it can be written in the form
$%\begin{equation}
    p(x)=\sum_{k=1}^{m} q_{k}^2(x)
$ %\end{equation}
for some polynomials $q_{k} \in \R[x]$, $k=1,\dots,m$.
%$\Sigma\R[x]^2=\left\{p \in \R[x]; \text{p  is SOS}\right\}$ denotes the cone of elements in $\R[x]$ that can be written as SOS polynomials.
The existence of an SOS decomposition is an algebraic certificate for nonnegativity of a polynomial. It is obvious that every SOS polynomial is nonnegative on $\R^n$. But the converse is not always true, that is, a nonnegative polynomial is not necessarily SOS. Dealing with positivity of a polynomial is hard but with SOS, it becomes easier as the problem boils down to semidefinite programming (SDP) or linear matrix inequalities (LMIs), a particular class of convex optimization problems for which efficient algorithms are available, as explained already in Section~\ref{sec:Intro}. For detailed accounts on SOS and positive polynomials and the algebraic concepts, we refer to \cite{Lau09, Lass15}.

%% file: converseLip.tex
\section{Converse Copositive Lyapunov Result}\label{sec:ConvCopLya}

In this section, we will establish an existence result for Lyapunov function, that is, if the system is exponentially stable then there exists a Lyapunov function, with certain properties, for such system.
There exist several results in the literature on converse Lyapunov theorems for systems where the vector fields are discontinuous, see \cite{DayaMart99, MasoBosc06} for switched systems, and \cite{CamlPang06} for complementarity systems. The results in \cite{DayaMart99, MasoBosc06} use linearity of the flows, and the results in \cite{CamlPang06} are restricted to the class of complementarity systems where the right-hand side is Lipschitz continuous (and hence not discontinuous). Here, we study the converse result where the flow maps are not necessarily linear, and the complementarity relations may induce discontinuities in the vector field. In essence, we generalize the converse results on differential inclusions presented in \cite{ClarLedy98, TeelPral00}. An essential difference compared to these results is that our system does not satisfy the regularity assumptions imposed in those works, and instead of strong stability, we address weak stability. Moreover, the structure of the system only allows construction over the admissible domain, which is a closed convex cone in our case. Our main result in this direction appears below. The proof of this theorem is a rather lengthy and technical affair and is carried out in the remainder of this section.

\begin{thm}\label{thm:converseMain}
Under Assumption~\ref{lip}, if the origin is globally exponentially stable for system~\eqref{eq:compSys}, then there exists a continuously differentiable cone-copositive Lyapunov function.
\end{thm}

%
%\section{Proof of Theorem~\ref{thm:converseMain}}
%\label{sec:proofConverse}
%
   
To prove Theorem~\ref{thm:converseMain}, we start with the following lemma.

\begin{lem}\label{lem:lipCont}
If Assumption~\ref{lip} holds and the origin is globally exponentially stable for system~\eqref{eq:compSys}, then there exists a globally Lipschitz function $\wh f:\R^n\to \R^n$ such that the system
\begin{equation}\label{eq:compSysLip}
\begin{aligned}
\dot x =& \wh f(x) + \eta \\
K^\star \ni \eta & \perp x \in K
\end{aligned}
\end{equation}
has a globally exponentially stable equilibrium and a continuously differentiable cone-copositive Lyapunov function $\wh V$. Moreover, $\wh V$ is a Lyapunov function for \eqref{eq:compSys}.
\end{lem}

The proof of Lemma~\ref{lem:lipCont} appears in Appendix~\ref{app:globLip}. Based on Lemma~\ref{lem:lipCont}, it can be assumed for the proof of Theorem~\ref{thm:converseMain}, without loss of generality, that $f$ in \eqref{eq:compSys} is a globally Lipschitz continuous vector field with modulus $L$ and this assumption is assumed to hold in the remainder of this section. Note that, in Theorem~\ref{thm:converseMain}, we assume the origin to be globally exponentially stable and our proof (appearing next) indeed uses that property. It remains to be seen if the proof can be adapted to the case where the origin is only asymptotically stable.

For the proof of Theorem~\ref{thm:converseMain}, we construct the Lyapunov function for \eqref{eq:compSys} by introducing a function $V : \R^n \to \R$, defined as
\begin{equation}\label{eq:defVConv}
V(z) = \int_{0}^\infty \left\| x(\tau; \proj_K(z)) \right\|^{\frac{2L}{\alpha}+1} \, d\tau,
\end{equation}
where $x(\tau;\overline z)$ denotes the solution to system \eqref{eq:compSys} at time $\tau \ge 0$ with $x(0^+) = \overline z \in K$. Note that $V$ is defined for each $z \in \R^n$ and not just for $z \in K$. When $z \not\in K$, the term $x(\tau; \proj_K(z))$ can be interpreted as the solution obtained by projecting the initial condition on $K$, and then propagating it continuously according to the system vector field. Thus, for $\tau > 0$, we have $x(\tau; \proj_K(z)) = x(\tau;z)$ for each $z \in \R^n$.

\subsection{Bounds on Solutions}

The following lemma demonstrates the continuity of solutions with respect to the initial conditions and plays an important role in the remainder of the proof.

\begin{lem}\label{lem:contIniCond}
Let $L$ be the Lipschitz modulus of $f$. If $x$ and $\wh x$ are two solutions to system \eqref{eq:compSys} that satisfy $x(0) = z \in K$ and $\wh x(0) = \wh z \in K$, then it holds that, for each $\tau > 0$,
\begin{subequations}\label{eq:contSolIniCond}
\begin{equation}
\| x(\tau; z) - \wh x(\tau; \wh z) \| \le e^{L \tau} \| z - \wh z\|  \label{eq:contSolIniConda}
\end{equation}
and for some $C > 0$,
\begin{equation}
\| x(\tau; z) \| \ge e^{-C \tau} \| z \|. \label{eq:contSolIniCondb}
\end{equation}
\end{subequations}
\end{lem}
\begin{proof}
It will be assumed without loss of generality that $z \in K$ and $\wh z \in K$ since $\| \proj_K(z) - \proj_K(\wh z) \| \le \| z -\wh z\|$.
By definition of the solution to \eqref{eq:compSys} and monotonicity of the normal cone operator, it follows that, for each $y \in K$,
\[
\left\langle \frac{d x}{dt} (t) - f(x(t)), y - x(t) \right \rangle \ge 0
\]
and similarly, for each $\wh y \in K$,
\[
\left\langle \frac{d \wh x}{dt} (t) - f(\wh x(t)), \wh y - \wh x(t) \right \rangle \ge 0,
\]
where we have suppressed the dependence of $x$ and $\wh x$ on the initial condition for brevity.
Letting $y = \wh x(t) \in K$, and $\wh y = x(t) \in K$, we get the following by adding the last two inequalities:
\[
\left\langle \frac{d}{dt} (x(t) - \wh x(t)), x(t) - \wh x(t) \right \rangle \le \left\langle f(x(t)) - f(\wh x(t)), x(t) - \wh x(t) \right \rangle,
\]
or equivalently,
\[
\frac{d}{dt} \| x(t) - \wh x(t) \|^2 \le 2 \left\langle f(x(t)) - f(\wh x(t)), x(t) - \wh x(t) \right \rangle.
\]
Because of the Lipschitz continuity assumption,
$\| f(x(t)) - f(\wh x(t))  \| \le L \| x(t) - \wh x(t)\|$,
and hence,
\[
 \frac{d}{dt} \| x(t) - \wh x(t) \|^2  \le 2 L \| x(t) - \wh x(t) \| ^2.
\]
The bound in \eqref{eq:contSolIniConda} now follows by integrating both sides, or invoking the so-called comparison lemma~\cite[Lemma~3.4]{Khalil02}. To get the bound in \eqref{eq:contSolIniCondb}, we make use of Proposition~\ref{prop:boundEta}  in the Appendix~\ref{app:basicsComp} which ensures that there exists a constant $C_\eta>0$ such that $\vert \eta \vert \le C_\eta \vert f(x(t))\vert $. We therefore get
\begin{align*}
\left\vert \frac{d}{dt} \| x(t) \|^2 \right\vert = 2 \left \vert \langle x(t), f(x(t)) + \eta \rangle\right\vert  \le 2 L(1+C_\eta) \| x(t)\|^2.
\end{align*}
In particular, $\frac{d}{dt} \| x(t) \|^2 \ge -2 L(1+C_\eta) \| x(t) \|^2$, and hence, the inequality in \eqref{eq:contSolIniCondb} follows by taking $C = L(1+C_\eta)$.
\end{proof}

To show that $V$ satisfies item 1) of Definition~\ref{def:funcLyap}, let us first use the bound in~\eqref{eq:contSolIniCondb} from Lemma~\ref{lem:contIniCond}, so that
\[
\begin{aligned}
V(z) & \ge \int_{0}^\infty e^{-(2L+\alpha)C\tau/\alpha}\| \proj_K(z) \|^{\frac{2L}{\alpha}+1} \, d\tau \\  & \ge \underline C \| \proj_K(z)\|^{\frac{2L}{\alpha}+1},
\end{aligned}
\]
for some $\underline C > 0$. Also, exponential stability of the origin implies that $\| x(\tau; z) \| \le c_0 e^{-\alpha \tau} \|\proj_K(z)\|$ and hence there exists $\overline C> 0$ 
\[
\begin{aligned}
V(z) & \le c_0 \int_{0}^\infty e^{-(2L+\alpha) \tau} \| \proj_K(z)\|^{\frac{2L}{\alpha}+1} \, d\tau \\
& \le  \overline C \,\| \proj_K(z)\|^{\frac{2L}{\alpha}+1}.
\end{aligned}
\]

%We will next show that $V$ is locally Lipschitz, and that it can be regularized to a differentiable function that satisfies the conditions listed in Definition~\ref{def:funcLyap}.

\subsection{Local Lipschitz Continuity of $V$}
To show that $V$ is locally Lipschitz continuous, we need the following two properties \cite{ClarSter93}:
\begin{itemize}
\item $V$ is continuous; and
\item its Dini subderivative\footnote{The Dini subderivative of $V$ at $x$ in the direction $v$ is defined as \[  DV(x;v):=\liminf_{w \to v, \eps \to 0^+}  \frac{V(x + \eps w) - V(x)}{\eps}.\]} satisfies 
\begin{equation}\label{eq:defDiniBnd}
D V(z;v) \le \phi (z) \| v \|
\end{equation}
for every $v \in \R^n$, every $z \in \R^n$, and some locally bounded function $\phi : \R^n \to \R$, with $\phi (z) > 0$ for $z \neq 0$.
\end{itemize}
The continuity of $V$ follows directly from Lemma~\ref{lem:contIniCond} as the exponential bound on the solutions of the system makes $V$ a composition of continuous functions.
These properties can again be shown using Lemma~\ref{lem:contIniCond}. Fix $v \in \R^n$. Consider a sequence of initial conditions $\wh z_k = z + \eps_k v$. We get
\begin{align}
& D V(z;v) \le \liminf_{\eps_k \to 0} \frac{V(z+\eps_k v) - V(z)}{\eps_k} \notag \\
& = \liminf_{\eps_k \to 0} \frac{1}{\eps_k} \left( \int_0^\infty ( \| \wh x_k(\tau; \wh z_k) \|^{\frac{2L}{\alpha}+1}  - \| x(\tau; z) \|^{\frac{2L}{\alpha}+1} ) d\tau \right) . \label{eq:boundDVlin}
\end{align}
Using the mean-value theorem, for each $s \ge 0$, there exists $\xi(s)$ between $\| \wh x_k(s;\wh z_k)\|$ and $\| x(s;z) \|$ such that 
\begin{align}
& \| \wh x_k(s; \wh z_k) \|^{\frac{2L}{\alpha}+1} - \| x(s; z) \|^{\frac{2L}{\alpha}+1} \notag \\
\le \ & \big \vert \,  \|\wh x_k(s;\wh z_k)\|^{\frac{2L}{\alpha}+1}  - \| x(s;z)\|^{\frac{2L}{\alpha}+1} \big \vert \notag \\
= \ & \big \vert \, \xi(s)^{2L/\alpha} (\| \wh x_k(s;\wh z_k)\| - \| x(s;z)\| ) \big \vert \notag \\
\le \ & \vert \xi(s) \vert^{2L/\alpha} \| \wh x_k(s;\wh z_k) - x(s;z) \| . \label{eq:bndMeanValue}
\end{align}
It follows from Lemma~\ref{lem:contIniCond} that $\| \wh x_k(s;\wh z_k) - x(s;z) \| \le e^{Ls} \eps_k \|v \|$. Substituting these bounds in \eqref{eq:boundDVlin}, we get
\[
D V(z;v) \le \| v\| \int_{0}^\infty  e^{Ls} \vert \xi(s)\vert^{2L/\alpha} \, ds .
\]
Due to the exponential stability assumption, $\| \xi(s) \| \le \wh c \, e^{-\alpha s} \| z \|$, for some $\wh c > 0$, and hence we choose
\[
\phi(z) =  \wh c \, \| z \|^{2L/\alpha} \int_0^\infty e^{-L s} \, ds,
\]
so that the bound \eqref{eq:defDiniBnd} is seen to hold. Thus, $V$ is locally Lipschitz continuous.
\subsection{Infinitesimal Decrease in $V$}
As the next step, we show that the function $V$ decreases along the system vector field. In what follows, we will denote the right-hand side of \eqref{eq:compSys} by $F(z)$, so that
\[
F(z) \in f(z) -\cN_K(z).
\]

The function $V$ in \eqref{eq:defVConv} is differentiable almost everywhere because it is locally Lipschitz continuous.
We next show that the Dini subderivative of $V$, along $F(z)$ is negative definite.
\begin{lem}\label{lem:calcGradae}
For the function $V:\R^n \to \R$ in \eqref{eq:defVConv}, and $z \in K$,
\begin{equation}\label{eq:lipDiniBound}
DV(z;F(z)) \le - \|z\|^{\frac{2L}{\alpha}+1}. %\left\langle \nabla V(z), F(z) \right\rangle \le -z^\top z.
\end{equation} 
\end{lem}
\begin{proof}[Proof of Lemma~\ref{lem:calcGradae}]
To prove \eqref{eq:lipDiniBound}, we need a bound on $V(z) - V(z+tF(z))$ for $t\ge 0$ sufficiently small.
We will get the desired bounds by rewriting the difference as
\begin{equation}\label{eq:gradEstTwoTerms}
V(z+tF(z)) - V(z) = \big[V(z+tF(z)) - V(x(t;z)) \big] 
+ \big[V(x(t;z)) - V(z) \big]
\end{equation}
and getting a bound on each of the two difference terms on the right-hand side.
The first term $V(z+tF(z)) - V(x(t;z))$ can be analyzed from the following lemma:
\begin{lem}\label{lem:discEst}
For $t> 0$ sufficiently small,  it holds that
\begin{equation}\label{eq:estDiscConverse}
\| z+tF(z)  - x(t;z) \|  \le o(t)
\end{equation}
for each $z \in K$.
\end{lem}
The proof of Lemma~\ref{lem:discEst} will follow momentarily. Using the estimate \eqref{eq:estDiscConverse}, and the inequalities \eqref{eq:boundDVlin} and \eqref{eq:bndMeanValue}, we get
\begin{align*}
V(z+tF(z)) -  V(x(t;z))  \le  C_\phi \| z+tF(z)  - x(t;z) \| = o(t),
\end{align*}
for a fixed $z \in K$, and some $C_\phi > 0$.
For the second term on the right-hand side of \eqref{eq:gradEstTwoTerms}, it follows from the definition of $V$ in \eqref{eq:defVConv}, with $x(0) = z$, that
\[
V(z) \ge \int_{0}^t \| x(\tau;z) \|^{\frac{2L}{\alpha}+1} d\, \tau+ \int_{0}^\infty \| x(\tau;x(t;z))\|^{\frac{2L}{\alpha}+1} \, d\,\tau
\]
and hence
\begin{equation}\label{eq:boundTempDiff}
V(x(t;z)) - V(z) \le - \int_{0}^t \| x(\tau;z) \|^{\frac{2L}{\alpha}+1} d\, \tau.
\end{equation}
Substituting the bounds from \eqref{eq:estDiscConverse} and \eqref{eq:boundTempDiff} in \eqref{eq:gradEstTwoTerms}, we get
\begin{align*}
& \liminf_{t \to 0^+} \frac{V(z+tF(z)) - V(z)}{t}  \\
\le \ & \liminf_{t \to 0^+}  \frac{- \int_{0}^t \| x(\tau;z) \|^{\frac{2L}{\alpha}+1} d\, \tau}{t} \\
= \ & \liminf_{t \to 0^+} - \| x(t;z)\|^{\frac{2L}{\alpha}+1} = - \|z\|^{\frac{2L}{\alpha}+1},
\end{align*}
and hence the Dini subderivative of $V$ is negative definite for almost every $z \in K$.
\end{proof}

\begin{proof}[Proof of Lemma~\ref{lem:discEst}]
By definition, the solution $x$ of system~\eqref{eq:compSys}, with $x(0) = z \in K$, satisfies
\begin{equation}\label{eq:xDotIneqLemma5}
\left\langle \dot x(t) - f(x(t)) , x(t) - y  \right\rangle \le 0, \quad \forall \, y \in K,
\end{equation}
for almost all $t \ge 0$.
For $h> 0$ small enough, introduce the function $\wt z: [0,h] \to \R^n$ given by
\[
\wt z (t) := z+t F(z) = z + t f(z) + t\eta_z,
\]
where $\eta_z$ is such that $\eta_z = 0$ for $z \in \inn (K)$ and $\eta_z \in \lccp(f(z), I, \cT_K(z))$ for $z \in \bd(K)$.
It is readily checked that $\wt z(t) \in K$ for all $t \in [0,h]$, and $\dot{\wt z} = \frac{d}{dt} \wt z(t) = F(z) = f(z)+ \eta_z$.
From the definition of $F(z)$, it follows that
$%\[
\left\langle f(z) - \dot{\wt z} \, , \, \wt y - z \right\rangle \le 0$, for all $\wt y \in K$, %\]
or equivalently,
\[
\left\langle f(z) - \dot{\wt z} \, , \wt y - \wt z(t) \right \rangle \le \left\langle F(z) - f(z), \wt z(t) -z \right\rangle, \, \forall \, \wt y \in K.
\]
For $t > 0$ small enough, we have $\left(x(t) - \wt z(t) + t \wt z(t) \right) \in K$. Since $K$ is a cone, we can take $\wt y = \frac{1}{t} \left(x(t) - \wt z(t) + t \wt z(t) \right) \in K$ to get
\[
\left\langle f(z) - \dot{\wt z} \, , x(t) - \wt z(t) \right \rangle \le t \left\langle \eta_z, \wt z(t) -z \right\rangle.
\]
Taking $y = \wt z(t)$ in \eqref{eq:xDotIneqLemma5}, and adding it to the last inequality, we get
\[
\left\langle \dot x(t) - \dot{\wt z}, x(t)-\wt z (t) \right\rangle \le \left \langle f (x(t)) - f(z),  x(t) - \wt z(t) \right\rangle 
+ t \left\langle \eta_z, \wt z(t) -z \right\rangle.
\]
To bound the terms on the right-hand side, we observe that
\begin{align*}
\| f(z) - f(x(t)) \| &\le L \| z - x(t) \| \\
& \le L \| z - \wt z(t) \| + L \| \wt z (t) - x(t) \| \\
& \le L t \| F(z) \| + L \| \wt z (t) - x(t) \|.
\end{align*}
Using Proposition \ref{prop:boundEta} in the Appendix~\ref{app:basicsComp}, there is some constant $C_{\eta}$ such that
\begin{align*}
\left\langle \eta_z, \wt z(t) -z \right\rangle & = \left\langle \eta_z, t F(z) \right\rangle \le t \, \| \eta_z \| \, \| F(z) \|  \le C_{\eta} t \| F(z) \|.
\end{align*}
Consequently, we get
\begin{align*}
& \quad \frac{1}{2} \frac{d}{d t} \| x(t;z) - \wt z (t) \| ^2 = \left\langle \dot{\wt z} - \dot x(t), \wt z (t) - x(t) \right\rangle \\
& \le \| f (z) - f(x(t)) \| \cdot \| \wt z (t) - x(t) \| + C_{\eta} t^2 \| F(z) \| \\
& \le L \| z - x(t) \| \cdot \| \wt z (t) - x(t) \| + C_{\eta} t^2 \| F(z) \| \\
& \le L t \| F(z) \|  \| \wt z (t) - x(t) \| + L \| \wt z (t) - x(t) \| ^2 + C_{\eta} t^2 \| F(z) \| \\
& \le C_{1} \| \wt z (t) - x(t) \| ^2 + C_{2,z} t^2
\end{align*}
where we used Young's inequality for the product term $L t \| F(z) \| \cdot \| \wt z (t) - x(t) \|$, and chose $C_1 = (L+0.5L^2)$ and $C_{2,z} = \max\{0.5\| F(z)\|^2, C_{\eta} \| F(z) \| \}$. Solving the differential inequality, and using the fact that, $\wt z(0) = x(0)$, we get
\[
\| \wt z(t) - x(t) \|^2 \le 2 C_{2,z} \int_0^t \exp(2 C_1(t-s))  s^2\, ds .
\]
Solving the integral on the right, we get
\[
\| \wt z(t) - x(t) \|^2 \le C_{3,z} \, t^3 + o(t^3),
\]
for some $C_{3,z} > 0$, whence the estimate in \eqref{eq:estDiscConverse} follows.
\end{proof}

\subsection{Regularization of $V$}
The next step is to regularize $V$ so that we obtain a continuously differentiable Lyapunov function.

\begin{lem}\label{lem:smoothReg}
Under Assumption~\ref{lip}, if the origin is globally exponentially stable for system~\eqref{eq:compSys}, then there exists a continuously differentiable cone-copositive Lyapunov function.
\end{lem}

\begin{proof}[Proof of Lemma~\ref{lem:smoothReg}]
Using the function $V$ in \eqref{eq:defVConv} as a template, we introduce
\begin{align*}
V_\sigma (z) & := \int_{\R^n} V (z-y) \psi_\sigma (y) dy \\
& = \int_{\R^n} V (\proj_K(z-y)) \psi_\sigma (y) dy
\end{align*}
where $\psi_\sigma$, $\sigma \in (0,1)$, is the so-called mollifier that satisfies: $\psi_{\sigma}\in \mathcal{C}^\infty(\R^n, \R_+)$, $\textrm{supp}(\psi_\sigma)\subset \B(0,\sigma)$, and $\int_{\R^n} \psi_\sigma(y) dy = 1$.
It follows from standard texts in functional analysis, see for example \cite[Proposition 4.21]{Brez10}, that $V_\sigma$ is continuously differentiable and for every $\eps > 0$ and a compact set $\cU_c$, there exists $\overline \sigma > 0$, such that for every $\sigma \in (0,\overline \sigma)$, we get $\vert V(x) - V_\sigma (x) \vert  < \eps$ for each $x \in \cU_c$. Next, we show that $\langle \nabla V_\sigma(z), F(z)\rangle$ approximates $D V(z,F(z))$, for $z \in K$.
Indeed, for a given $y \in \R^n$, and $z \in K$, let $\bar z_y = \proj_K(z-y)$. It then follows that\footnote{Since $V$ is locally Lipschitz, its gradient $\nabla V$ exists almost everywhere and the value of the integral on the right-hand side is not affected by the value of $\nabla V$ on a set of Lebesgue measure zero.}
\begin{align*}
\left \langle \nabla V_\sigma(z), F(z) \right \rangle & = \int_{\R^n} \left \langle \nabla V (\bar z_y), F(\bar z_y) \right \rangle \psi_\sigma (y) dy \\
& \quad  + \int_{\R^n} \left \langle \nabla V (\bar z_y), F(z) - F(\bar z_y) \right \rangle \psi_\sigma (y) dy \\
& \le - \| z\|^{\frac{2L}{\alpha}+1} + \eps + C \int_{\B(0,\sigma)} \! \! \! \! \| \nabla V (\bar z_y) \| \| y \| dy
\end{align*}
where the bound on the first integral is due to Lemma~\ref{lem:calcGradae}, and the bound on the second integral is obtained from the Lipschitz continuity of $f$ and that of $\eta$ given in Proposition~\ref{prop:boundEta} in Appendix~\ref{app:basicsComp}. Thus, on each compact set excluding the origin, we can find a function $V_\sigma$ such that $\left \langle \nabla V_\sigma(z), F(z) \right \rangle$ is negative definite.

Let us now consider $\{\cU_i\}_{i\in \Nz}$ to be a locally finite open cover of $\R^n\setminus \{ 0 \}$ with $\cU_i$ bounded and $0 \not \in \cl(\cU_i)$, for each $i \in \Nz$. Let $\{\chi_i\}_{i \in \Nz}$ be a subordinated $\cC^1$ partition of unity. For each $i \in \Nz$, and $\eps_i > 0$, we can choose the function $V_i$ such that $\vert V(x) - V_i(x) \vert < \eps_i$, and $\left \langle \nabla V_i(x), F(x) \right \rangle$ is negative, for each $x \in \cl(\cU_i)$. Let $\overline V:\R^n \to \R_+$ be such that $\overline V(0) = 0$ and $\overline V(x) := \sum_{i \in \Nz} \chi_i(x) V_i(x)$ for $x \neq 0$, then following the analysis in \cite[Pages~106-108]{ClarLedy98}, it is seen that $\overline V$ is a cone-copositive Lyapunov function which is $\cC^1$ on $\R^n \setminus \{0\}$, and continuous at $\{0\}$. Finally, to achieve differentiability at the origin, we can introduce a positive definite function $\beta : \R_+ \to \R_+$ with $\beta'(s) > 0$ for each $s > 0$ such that $W(x) = \beta(\overline V(x))$ is a continuously differentiable cone-copositive Lyapunov function with respect to $K$.
\end{proof}

\begin{rem}
The construction given in the proof of Lemma~\ref{lem:smoothReg} actually gives a $\cC^\infty(\R^n, \R)$ Lyapunov function. This regularization technique is inspired by \cite{ClarLedy98}, and has also been used for smoothening of locally Lipschitz Lyapunov functions for hybrid systems \cite[Chapter~7]{GoebSanf12} and switched systems \cite{DellaTanw19}.
\end{rem}

%% file: homoPoly.tex
% !TEX root = sosCertEVIs.tex
%%%%%%%%%%%%%%%%%%%%%%%%%%%%%%%%%%%

\section{Homogeneous and Polynomial Lyapunov Approximations}\label{HomPolLyap}

For numerical purposes, it is useful to show the existence of homogeneous Lyapunov functions, and if possible, polynomial Lyapunov functions. In this section, we address the question whether there exist Lyapunov functions with these additional properties.

\subsection{Homogeneous Lyapunov Functions}

First, we show that the previous developments can be generalized to construct a homogeneous Lyapunov function when the vector field $f$ in the system description~\eqref{eq:compSys} is homogeneous. 

\begin{defn} 
The vector field $f$ is homogeneous of degree $d \ge 1$ if
\[
f(\lambda x) = \lambda^d f(x).
\]
for each $x \in K$ and $\lambda\geq 0$.
\end{defn}

The next two statements are generalizations of results given in \cite{Ros92}.

\begin{prop}
Under Assumption \ref{lip}, if the origin is locally exponentially stable for system \eqref{eq:compSys} with $f$ homogeneous, then it is also globally exponentially stable. 
\end{prop}

\begin{proof}
We first show that if $x:[0,\infty) \to K$ is a solution that satisfies \eqref{eq:compSys} starting with initial condition $x_0$, then for each $\lambda \ge 0$ and $t \ge 0$, the function $y(t) = \lambda x(\lambda^{d-1} t)$ is also a solution to system \eqref{eq:compSys} starting with initial condition $\lambda x_0$.
It follows by inspection that $y(t) \in K$, for each $t \ge 0$. Noting that for each $z\in K$, and $\lambda > 0$, there exists $\overline z \in K$ such that $z = \lambda \overline z$, we get
\begin{align*}
& \quad \left\langle \dot y(t) - f(y(t)), z - y(t) \right\rangle \\
 & = \left\langle \lambda^{d} \dot x(\lambda^{d-1} \, t) - f(\lambda x(\lambda^{d-1} \, t)), z -  \lambda x(\lambda^{d-1} \, t) \right\rangle \\
& = \lambda^{d} \left\langle \dot x(\lambda^{d-1} \, t) - f(x(\lambda^{d-1} \, t)), \lambda \overline z -  \lambda x(\lambda^{d-1} \, t) \right\rangle \\
& = \lambda^{d+1} \left\langle \dot x(\lambda^{d-1} \, t) - f(x(\lambda^{d-1} \, t)), \overline z -  x(\lambda^{d-1} \, t) \right\rangle \ge 0,
\end{align*}
and hence $\dot y(t) - f(y(t)) \in -\cN_K(y(t))$ for almost every $t \ge 0$.

Since the origin is locally exponentially stable, there is an open set relative to $K$, say $\cR_0$, such that for each $x(0) \in \cR_0$, the corresponding solution $x$ converges to the origin. For an initial condition $y(0) \not \in \cR_0$, there is a constant $\lambda > 0$ such that $y(0) = \lambda x(0)$, with $x(0)\in \cR_0$. Since the solutions are unique, the above reasoning shows that the solution starting from $y(0)$ stays within a bounded set and converges to the origin.
\end{proof}

The next result allows us to construct a homogeneous Lyapunov function under local exponential stability. The proof is inspired from \cite{Ros92} and is provided in Appendix~\ref{app:proofHomo} just for the sake of completeness.
\begin{prop}\label{prop:homogLyapfct}
Consider dynamical system \eqref{eq:compSys} with $f$ homogeneous and the origin locally exponential stable. Let $W \in \cC^\infty(\R^n, \R)$ be a  cone-copositive Lyapunov function for \eqref{eq:compSys}. Let $a \in \cC^\infty (\R,\R)$ be such that 
\begin{equation}
   a=
    \begin{cases}
      0 & \text{on}\ (-\infty,1], \\
      1 & \text{on}\ [2,\infty),
    \end{cases}
  \end{equation}
 and $\nabla a(s) \ge 0$, for each $s \in \R$. Let $k$ be a positive integer. Then the function
 \begin{equation}
   \overline{W}(x)=
    \begin{cases}
      \int_{0}^{\infty} \frac{1}{\;\lambda^{k+1}} (a \circ W)(\lambda x)  \, d\lambda & \text{if}\ x \in \R^n\backslash\lbrace{0}\rbrace, \\
      0 & \text{if}\ x=0,
    \end{cases}
  \end{equation}
is a cone-copositive Lyapunov function of class $\cC^{k-1}$ on $\R^n\backslash\lbrace{0}\rbrace$, and it satisfies
\[
\overline{W}(s x)=s^k\overline{W}(x) 
\]
 for all $x \in \R^n\backslash\lbrace{0}\rbrace$ and $s > 0$.
\end{prop}

\subsection{Polynomial Approximation}

For the class of numerical algorithms that we will propose in the next section, it is important to show that the cone-copositive Lyapunov functions of \eqref{eq:compSys} can actually be approximated by polynomial functions. Among the existing results in this direction, it is seen that the existence of polynomial Lyapunov functions has been shown under certain restrictions only. In \cite{Peet09}, the authors use generalizations of the Weierstrass approximation theorem for nonlinear systems with smooth vector fields to show existence of polynomial Lyapunov functions on compact sets for exponentially stable systems. In case of switched systems,  the existence of polynomial Lyapunov functions has been proven in \cite{MasoBosc06} when the solution maps (parameterized by time) are linear functions of the initial condition. 
%In addition, there exist examples of asymptotically stable nonlinear systems which do not admit a polynomial Lyapunov function \cite{AhmKrsPar11}.
Such methods cannot be generalized here because our vector fields are not even continuous, and even with $f$ linear in \eqref{eq:compSys}, the resulting solution maps for the complementarity systems are nonlinear and hence nonconvex. As an example of this last observation, we consider the following example:

\begin{exam}[Constraints make the solution space nonlinear]
Let $f(x) = Ax$ with $A = \begin{bsmallmatrix}0 & 1\\-1 & 0\end{bsmallmatrix}$ and $K = \R_{+}^2$ and let $x_1(0) = (a, 0)^\top$ and $x_2(0) = (0,b)^\top$. Let $x_i:\R\to \R^2$ be the solution starting with initial condition $x_i(0)$, $i = 1,2$, and $z$ denote the solution starting with initial condition $x_1(0) + x_2(0)$. It can be checked that $z(t)$ does not equal $x_1(t) + x_2(t)$, for any $t > 0$ because we have $x_1(t)=x_1(0)$ for $t \ge 0$, $x_2(t)=e^{At}x_2(0)=(b\sin(t), b\cos(t))$ which gives us $x_1(t)+x_2(t)=(a+b\sin(t), b\cos(t))$, but we have $z(t)=e^{At}(x_1(0)+x_2(0))=(a\cos(t)+b\sin(t), -a\sin(t)+b\cos(t))$ which is not equal to $x_1(t)+x_2(t)$ for $a,b \ne 0$.
\end{exam}

These discussions and the example suggest that it may not be possible to find a homogeneous polynomial approximation to the Lyapunov function proposed in Theorem~\ref{thm:converseMain}. Due to lack of any known results on density of homogeneous polynomials in the class of differentiable functions, we enlarge our search to rational functions whose numerator and denominator are homogeneous polynomials. For such functions, we have the following density result \cite[Lemma~2.1]{AhmKha19}:

%To show that an exponentially stable system \eqref{eq:compSys} can admit a polynomial Lyapunov function, we have to introduce the following additional assumption:

%\begin{assump}
%\item\label{ass:convexFlow} The solution map of the system~\eqref{eq:compSys} has the property that, for each $t \ge 0$ and $z \in K$, the mapping
%\[
%z \mapsto \vert x(t;z) \vert^2
%\]
%from $K$ to $\R$ is strictly convex.
%\end{assump}
%Obviously, under \ref{ass:convexFlow}, the function $V$ proposed in the previous section is convex, and this convexity is retained by the function $W$ obtained by regularizing $V$.

%As the final step to get a polynomial Lyapunov function from $W$, we introduce the following lemma which shows that a homogeneous function can be approximated by a homogeneous rational function where the numerator and the denominator are polynomials.

\begin{prop} \label{prop:PolynApprox}
Let $W \in \cC^1(\R^n;\R_+)$ be a homogeneous function of degree $d$ and $\epsilon >0$ be a given scalar. There exist an even integer $r$ and a homogeneous polynomial $p$ of degree $r+d$, such that
\[
\max \left\{\max_{x\in S^{n-1}} \left | \wt W(x) \right |, \max_{x\in S^{n-1}} \left\| \nabla \wt W(x)\right\| \right\} \le \epsilon
\]
where $S^{n-1}$ denotes the unit sphere in $\R^n$ and $\wt W(x) = W(x) - \frac{p(x)}{\|x\|^r}$.
\end{prop}

With such a rational function in hand which approximates the homogeneous function from Proposition~\ref{prop:homogLyapfct} (in terms of value and gradient) to desired accuracy, one can establish the existence of a rational homogeneous cone-copositive Lyapunov function.
%\begin{lem}\label{prop:polyLyap}
%Consider a function $W:\R^n \to \R$ which is continuously differentiable, strictly convex and homogeneous of degree $d$. Let $f:\R^n \to \R^n$  be locally bounded such that $\left\langle \nabla W (x), f(x) \right\rangle < 0$.
%Then there exists a polynomial $\wt W: \R^n \to \R$, of sufficiently high degree, such that,
%\[
%\left\langle \nabla \wt W (x), f(x) \right\rangle < 0.
%\]
%\end{lem}

%% file: sosAlgos.tex
% !TEX root = sosCertEVIs.tex
%%%%%%%%%%%%%%%%%%%%%%%%%%%%%%%%%%%
\section{Numerical Construction Using Convex Optimization}\label{sec:comp}

In the previous section, we motivated the need for computing cone-copositive homogeneous Lyapunov functions for the class of constrained dynamical systems \eqref{eq:compSys}. Proposition~\ref{prop:PolynApprox} suggests that for a certain class of complementarity systems, we can reduce our search of Lyapunov functions to the space of rational polynomial functions, where the denominator has a certain structure. By fixing the denominator, we reformulate our problem as finding the numerator in the form of polynomial which satisfies certain inequalities. We carry out the steps by specifying the inequalities that need to be satisfied, and the algorithms using convex optimization methods that can be implemented for computing such functions.

Just as a quick motivation for what follows, we remark that contrary to unconstrained linear systems, the following example shows that copositive Lyapunov functions cannot be simply obtained by solving a linear equation, and hence there is a need to develop tools for computing them.
\begin{exam}[Copositive Lyapunov functions are not obtained by solving linear equations]
Let $K=\R_{+}^2$ and
$A = \begin{bsmallmatrix} -1 & -2\\ -1 &-1\end{bsmallmatrix}$.
Let $H = \begin{bsmallmatrix} 1 & 0\\ 0 &1\end{bsmallmatrix}$
the identity matrix which is copositive on cone $K$. By solving the equation $A^\top G + GA =-H$, we obtain $G = \begin{bsmallmatrix} -1 & \frac{3}{4}\\ \frac{3}{4} & -\frac{1}{4}\end{bsmallmatrix}$ which is not copositive.
On the other hand, if we take for example the copositive matrix $\widetilde{H} = \begin{bsmallmatrix} 1 & 2\\ 2 & 1\end{bsmallmatrix}$, we obtain the copositive matrix $\widetilde{G} = \begin{bsmallmatrix} 1 & -\frac{1}{4}\\ -\frac{1}{4} & \frac{3}{4} \end{bsmallmatrix}$ by solving $A^\top \widetilde{G} + \widetilde{G}A =-\widetilde{H}$.

This example shows that for a given matrix $A$, we can have a copositive matrix $H$ without the existence of $G$ copositive verifying $A^\top G + GA =-H$ , but with existence of some $\widetilde{G}$ such that $-A^\top \widetilde{G}- \widetilde{G} A$ is copositive.
\end{exam}

We now establish the inequalities which will be used in our algorithms to find copositive homogeneous Lyapunov functions. We restrict our attention to full-dimensional polyhedral cones, that is, $K=\left\{x\in \R^n \vert Fx \ge 0\right\}$ with non-empty interior. By using Proposition~\ref{prop:PolynApprox}, let
\[
V(x) = \frac{h(x)}{(\sum_{i=1}^{n} x_i^2)^r} = \frac{h(x)}{\|x\|_{2}^{2r}}
\]
where $r$ is a nonnegative integer, and $h(\cdot)$ is a homogeneous polynomial of degree at least $2r+1$, copositive on $K$. Here, we used the notation that $x=(x_{1},x_{2},\dots,x_{n})^\top \in K$.

As we know, finding such Lyapunov function is equivalent to finding $V$ that satisfies the inequalities:
\begin{subequations} 
\begin{gather}
V(x)= \frac{h(x)}{\|x\|_{2}^{2r}} \ge 0, \enspace \forall x \in K \backslash \lbrace 0 \rbrace \\
-\left\langle \nabla V(x), f(x)+ \eta \right\rangle \ge 0, \enspace \forall x \in K \backslash \lbrace 0 \rbrace
\end{gather}
\end{subequations}
where
\begin{align*}
-\left\langle \nabla V(x), f(x)+ \eta \right\rangle 
 = \frac{-\|x\|_{2}^2 \left\langle \nabla h(x), f(x)+ \eta \right\rangle +2rh(x) \left\langle x, f(x)+ \eta \right\rangle}{\|x\|_{2}^{2(r+1)}}.
\end{align*}
with $\eta  = \lccp(f(x), I, \cT_K(x))$.
The numerator is denoted by
\[
s(x) :=-\|x\|_{2}^2 \left\langle \nabla h(x), f(x)+ \eta \right\rangle +2rh(x) \left\langle x, f(x)+ \eta \right\rangle
\]
which is a homogeneous polynomial if $h$ and $f$ are homogeneous polynomials.
So we have
\begin{subequations} 
\begin{gather}
V(x)= \frac{h(x)}{\|x\|_{2}^{2r}} \ge 0, \enspace \forall x \in K \backslash \lbrace {0} \rbrace \\
-\left\langle \nabla V(x), f(x)+ \eta \right\rangle = \frac{s(x)}{\|x\|_{2}^{2(r+1)}} \ge 0, \enspace \forall x \in K \backslash \lbrace {0} \rbrace.
\end{gather}
\end{subequations}
Thus, finding a copositive $V$ for system~\eqref{eq:compSys} with the structure imposed in Proposition~\ref{prop:PolynApprox} boils down to finding $h$ and $s$ such that
\begin{subequations} \label{stabcondts}
\begin{gather}
h(x) \ge 0, \enspace \forall x \in K \backslash \lbrace {0} \rbrace \\
s(x) \ge 0, \enspace \forall x \in K \backslash \lbrace {0} \rbrace.
\end{gather}
\end{subequations}
Since $\eta$ is nonzero only on the boundaries of $K$, we replace the second inequality in \eqref{stabcondts} by inequalities with respect to each face of polyhedron $K$.
Let $S_i := \{ x \in K \, \vert \, (Fx)_i = 0 \}$, $i \in \left\{1,\dots,n_K\right\}$ denote the faces of $K$. Let 
\[
s_i(x)=-\|x\|_{2}^2 \left\langle \nabla h(x), f(x)+ \eta_i \right\rangle +2rh(x) \left\langle x, f(x)+ \eta_i \right\rangle
\]
for all $x \in S_{i}$ where $\eta_{i}  = \lccp(f(x), I, \cT_K(x))$. In the interior of  $K$, we have $\eta=0$ so let 
\[ s_0(x)=-\|x\|_{2}^2 \left\langle \nabla h(x), f(x) \right\rangle +2rh(x) \left\langle x, f(x) \right\rangle.
\]
Consequently, the inequalities used for finding $V$ can be rewritten as follows:
\begin{subequations} \label{Algoineq}
\begin{gather}
 h(x) \ge 0, \enspace \forall x \in K \backslash \lbrace {0} \rbrace \\
 s_{0}(x) \ge 0, \enspace \forall x \in \inn (K \backslash \lbrace {0} \rbrace)\\
 s_{i}(x) \ge 0, \enspace \forall x \in S_{i}, \enspace i \in \left\{1,\dots,n_K\right\}.
\end{gather}
\end{subequations}

While computing cone-copositive Lyapunov functions $V$ using inequalities \eqref{Algoineq}, we notice that we are faced with two problems, which prevent us from using conventional SOS techniques. The first problem is that there is no readily available Positivstellensatz for unbounded domains like cones. The second problem is that our Lyapunov functions are not necessarily SOS because a SOS polynomial is in particular positive definite but our systems require searching for a Lyapunov function beyond positive definite functions.

To overcome these problems, we study two techniques for finding polynomials that satisfy \eqref{Algoineq}. In what follows, we assume that $K = \R_+^n$, that is, $K$ is the positive orthant with $n$ faces so that $S_i = \{x \in \R_+^n \, \vert \, x_i = 0\}$, $i = 1, \dots, n$. The more general case of polyhedral cones can be covered with state transformations or decompositions but such details are being avoided for the ease of exposition.
\subsection{Discretization Method}
The basic idea behind the discretization methods is to select a certain number of points in the cone $\R_+^n$ and evaluate the inequalities \eqref{Algoineq} with a certain polynomial function parameterized by finitely many unknowns.
This allows us to construct an inner approximation of copositive polynomials with respect to cone $\R_+^n$.
%Discretization methods are adapted to present an approximation only in the region of the cone that is applicable for the optimization. 
In the literature, there exist algorithms for checking copositivity of a matrix using discretization methods \cite{BunDur08,BunDur09}, \cite{Dur10}, and using a moment relaxation hierarchy \cite{NiYaZha18}. Here, we restrict ourselves to discretization schemes and generalize the existing algorithms for arbitrary nonlinear polynomials (not necessarily quadratic functions).% over an arbitrary semi-algebraic cone $K$ (not necessarily $\R^n_+$).

To describe this discretization algorithm, let us first consider the convex cone of copositive polynomials 
\begin{equation} \label{conecop}
\cC:=\left\{
g \in \R^d [x] \Bigg \vert\ \begin{aligned} & g \text{ is homogeneous and} \\ & g(x) \ge 0 \enspace \text{for all} \enspace x \in \R_+^n 
\end{aligned}\right\},
\end{equation}
where $\R^d[x]$ denotes the ring of polynomials of degree $d$, over the field of reals, in $x \in \R^n$.
We will establish an inner approximation of $\cC$ based on simplicial partitions inside cone $\R_+^n$. To do so, we first need to introduce {\em tensors}, which generalize the notion of a matrix, and will be used for compact representation of polynomials of our interest.

\begin{defn}
A tensor $\cA$ of order $d$ over $\R^n$ is a multilinear form
\begin{equation*}
\begin{array}{ccccc}
& \underbrace{\R^n \times \R^n \times \dots \times \R^n}_{d \text{ times}} & \to & \R \\
 & (x^1,x^2,\dots,x^d) & \mapsto & \cA[x^1,x^2,\dots,x^d] \\
\end{array}
\end{equation*}
where 
\[
\cA[x^1,x^2,\dots,x^d]= \sum_{i_1=1}^n \sum_{i_2=1}^n \sum_{i_d=1}^n a_{i_1,i_2,\dots,i_d} x^1_{i_1} \cdots x^d_{i_d}
\]
and $a_{i_1,i_2,\dots,i_d}$ corresponds to a real number from a table with $n^d$ entries, indexed by $i_1,i_2,\dots,i_d \in \{1, \dots, n\}$. We say that $\cA$ is symmetric if
\[
a_{i_1,i_2,\dots,i_d} = a_{j_1,j_2,\dots , j_d}
\]
whenever $i_1+i_2+\dots+i_d = j_1+j_2+\dots+j_d$, for all possible permutations $i_1,i_2,\dots,i_d$ and $j_1,j_2,\dots , j_d$ of $\{1, \dots, n\}$.
\end{defn}
A classic matrix $A \in \R^{n \times n}$ describes a tensor of order $2$ over $\R^n$, also called a quadratic form, where the coefficients of the quadratic form belong to a table with $n^2$ entries $a_{i,j}$ with $i,j=\{1,\dots, n\}$. 
A general homogeneous polynomial $g \in \R^d[x]$, with $d \ge 2$, can be written as
\[
g(x)=g(x_1,\dots,x_n)=\sum \limits_{\underset{i_{1}+\dots+i_{n}=d}{i=(i_1, \dots, i_n)}} a_{i} x_{1}^{i_1} \cdots x_{n}^{i_n}.
\]
Using the tensor representation, $g$ can also be compactly written in the form
\begin{equation}\label{tensorpol}
g(x)=G[\underbrace{x,x,\dots,x}_{\text{$d$ times}}]
\end{equation}
where $G$ is a symmetric tensor. The following lemma shows an equivalent expression for copositivity which we will consider all along this section.
\begin{lem}\label{lem:norm1}
Consider a homogenous polynomial $g \in \R^d[x]$ of degree $d$ and let $\| \cdot \|$ denote any norm in $\R^n$. We have
\[
g \in \cC \iff g(x) \ge 0 \enspace \text{for all} \enspace  x \in \R_+^n \enspace \text{with} \enspace \|x\|=1.
\]
\end{lem}

\begin{proof} $[\Rightarrow]$ is obvious. $[\Leftarrow]$: Take $x \in \R_+^n$ with $\|x\| \ne 1$. If $\|x\|=0$ then $x=0$ and $g(0)=0$ because of the homogeneity of $g$. If $\|x\|>0$ then $\widetilde{x}:=\frac{x}{\|x\|}$ fulfills $\|\widetilde{x}\|=1$, therefore $g(x)=g(\|x\| \widetilde{x})=\|x\|^d g(\widetilde{x}) \ge 0$, for all $x \in \R_+^n$ which means $g \in \cC$.
\end{proof}
If we choose the 1-norm, then the set $\Delta^S := \left\{x \in \R_+^n \vert\ \|x\|_{1}=1\right\}$ is the standard simplex. %If $K = \R^n_+$, then the vertices of $\Delta^S$ are the unit vectors $e_{1},\dots,e_{n}$.
Because of Lemma~\ref{lem:norm1}, copositivity of a homogenous polynomial $g$ is then expressed as
\[
g(x) \ge 0 \enspace \text{for all} \enspace x \in \Delta^S.
\]

Our goal is to discretize the simplex $\Delta^S$ and obtain a hierarchy of linear inequalities with respect to the discretization points which allow us to approximate the set $\cC$.

\begin{defn} \label{simplicialpartition}
Let $\Delta$ be a simplex\footnote{An $n$-simplex is an $n$-dimensional polytope which is the convex hull of its $n+1$ vertices $\{x_0, x_1, \ldots, x_n\}$, namely \[ \Delta := \left\{\theta_0 x_0 + \dots \theta_n x_n \bigg \vert\ \sum_{i=0}^{n} \theta_i =1 \enspace \text{and} \enspace \theta_i \ge 0 \enspace \text{for all} \enspace i \in \{0,\ldots,n\}\right\}. \]} in $\R^n$. A family $\cP_m :=\left\{\Delta^1,\dots,\Delta^m\right\}$ of simplices satisfying
\[
\Delta=\bigcup_{i=1}^m \Delta^i \enspace \text{and} \enspace \inn\Delta^i \cap \inn\Delta^j=\emptyset \enspace \text{for} \enspace i \ne j
\]
is called a simplicial partition of $\Delta$.
\end{defn}

\begin{defn} \label{diametersimplex}
For a simplicial partition $\cP_m=\left\{\Delta^1,\dots,\Delta^m\right\}$ of a simplex $\Delta$, where $v_{1}^k,\dots,v_{p}^k$ denote the vertices of simplex $\Delta^k$, the maximum diameter of a simplex in $\cP_m$ is defined as
\[
\delta(\cP_m):=\max_{k \in \left\{1,\dots,m\right\}} \max_{i,j \in \left\{1,\dots,p\right\}} \|v_{i}^k - v_{j}^k\|.
\]
\end{defn}

%For a simplicial partition $\cP_m=\{\Delta^1,\dots,\Delta^m\}$ of $\Delta^S$, we let $V_{\cP_m}$ denote the set of all vertices of simplices in $\cP_m$, and $E_{\cP_m}$ the set of all edges of simplices in $\cP_m$. The cardinality of $V_{\cP_m}$ is $p_m:=|V_{\cP_m}|$.
For a given partition $\cP_m=\left\{\Delta^1,\dots,\Delta^m\right\}$ of $\Delta^S$ and a homogeneous polynomial $g$ defined as in \eqref{tensorpol}, let us consider the set $Q^k$, which contains all the vertices of $\Delta^k$, and moreover, let the set $\cI_{\cP_m}^{p_m,d}$ be defined as
\begin{equation}\label{innerapprox}
\cI_{\cP_m}^{d} =\left\{\begin{aligned} & g \in \R^d[x] \, \Big\vert \,  G[q_1,q_2,\dots,q_d] \ge 0, \\
& \{q_1,q_2,\dots,q_d\} \in Q^k, k = 1,\dots,m \end{aligned} \right\}.
\end{equation}

The following proposition shows that $\left\{\cI_{\cP_l}^{p,d}\right\}_{l \in \Nz}$ is a sequence of inner approximation which approximates the cone of copositive polynomials under the condition that the diameter of the simplicial partition goes to zero.
\begin{prop} \label{prop:InnerApp}
Let $\left\{\cP_{l}\right\}_{l\in \Nz}$ be a sequence of simplicial partitions of $\Delta^S$ such that $\delta(\cP_{l}) \rightarrow 0$. Then, we have
\[
\inn \cC \subseteq \bigcup_{l \in \Nz} \cI_{\cP_{l}}^{d} \subseteq \cC, \enspace \text{and hence} \enspace \cC=\overline{\bigcup_{l \in \Nz} \cI_{P_{l}}^{d}}.
\]
\end{prop}

Proposition~\ref{prop:InnerApp} ensures that if we construct a hierarchy of linear programs by making the partition finer, we can find a rational polynomial Lyapunov function for homogenous systems if the origin is exponentially stable. A pseudocode for implementing this method is given in the form of Algorithm~\ref{alg:Disc} in Appendix~\ref{app:algoDisc}.

To prove Proposition~\ref{prop:InnerApp}, we need the following two lemmas. The first one gives us sufficient conditions for copositivity and the second one a necessary condition for strict copositivity. 

\begin{lem} \label{simpconv}
Consider the set of vectors, $V_{\cP} := \{ v_1, \cdots, v_p\}$, and let $\Delta=\mathrm{conv}\{v_{1},\dots,v_{p}\}$. If
\begin{equation} \label{assum}
G[v_{i_1},v_{i_2},\dots,v_{i_d}] \ge 0 \enspace \text{for all} \enspace i_1,i_2,\dots,i_d \in \{1,\dots,p\},
\end{equation}
then $g(x)= G[x,x,\dots,x] \ge 0$ for all $x \in \Delta$.
\end{lem}

\begin{proof}
For each point $x \in \Delta$, we can represent it in the affine hull of $\Delta$ by its uniquely determined barycentric coordinates $\lambda=(\lambda_{1},\dots,\lambda_{p})$ with respect to $\Delta$ i.e.
\[
x=\sum_{j=1}^{p} \lambda_{j}v_{j} \enspace \text{with} \enspace \sum_{j=1}^{p} \lambda_{j}=1.
\]
This gives
\begin{align*}
g(x)
&=G[x,x,\dots,x]\\
&=G\big[\sum_{i_1=1}^{p} \lambda_{i_1}v_{i_1},\sum_{i_2=1}^{p} \lambda_{i_2}v_{i_2},\dots,\sum_{i_d=1}^{p} \lambda_{i_d}v_{i_d}\big]\\
&=\sum_{i_1,i_2,\dots,i_d=1}^{p}G[v_{i_1},v_{i_2},\dots,v_{i_d}]\lambda_{i_1}\lambda_{i_2} \dots \lambda_{i_d}.
\end{align*}
For $x \in \Delta$, we have $\lambda_i \ge 0$, and by the assumption \eqref{assum}, we get $g(x) \ge 0$ for all $x \in \Delta$.
\end{proof}

\begin{lem} \label{strictcop}
Let $g \in \R^d[x]$ be strictly copositive and homogeneous. Then there exists $\epsilon >0$ such that for any finite simplicial partition $\cP_m=\{\Delta^1,\dots,\Delta^m\}$ of $\Delta^S$ with $\delta(\cP_m) \le \epsilon$, we have $\forall  k=1,\dots,m,$ and $i_1,i_2,\dots,i_d \in \{1,\dots, \vert Q^k \vert\}$,
\[
G[v_{i_1}^k,v_{i_2}^k,\dots,v_{i_d}^k] >0,
\]
where $v_{1}^k,v_2^k, \cdots \in Q^k$, the set containing the vertices of the simplex $\Delta^k$.
\end{lem}

\begin{proof}
We have by assumption that $g$ is strictly copositive which means that the tensor form $G[x^1,x^2,\dots,x^d]$ is strictly positive on the diagonal of $\Delta^S\times\Delta^S\times \dots \times\Delta^S \subset \R^{nd}$. By continuity, for every $x^i \in \Delta^S$, there exists $\epsilon_{x^{i}} > 0$ such that, for $j = 1,\dots, d$,
\[
\|x^i-x^j\| \le \epsilon_{x_{i}} \Rightarrow G[x^1,x^2,\dots,x^d] > 0.
\]
Since $G$ is uniformly continuous on the compact set $\Delta^S\times \dots \times\Delta^S$, it follows that $\epsilon:=\inf_{x^i \in \Delta^S} \epsilon_{x^{i}}$ is strictly positive.

Let $\cP_m=\{\Delta^1,\dots,\Delta^m\}$ be a simplicial partition of $\Delta^S$ with $\delta(\cP_m) \le \epsilon$. Let $\Delta^k$ with $k=1,\dots,m$ be an arbitrary simplex, and $v_{i}^k$, $i=1,\dots, \vert Q^k\vert$ arbitrary vertices of $\Delta^k$. Then, for $i,j=1,\dots,\vert Q^k\vert$, we have $\|v_{i}^k-v_{j}^k\| < \epsilon$, and therefore $G[v_{i_1}^k,v_{i_2}^k,\dots,v_{i_d}^k] > 0$ for all $i_1,i_2,\dots,i_d \in \{1,\dots, \vert Q^k\vert\}$, so the statement is proved.
\end{proof}

\begin{proof}[Proof of Proposition~\ref{prop:InnerApp}]
Take $g \in \inn \cC$, which means that $g$ is strictly copositive. Lemma~\ref{strictcop} implies that there exists $l_{0} \in \Nz$, such that $g \in \cI_{\cP_{l_0}}^{p,d}$. Then $g \in \bigcup_{l \in \Nz} \cI_{\cP_{l}}^{d}$, and $\inn \cC \subseteq \bigcup_{l \in \Nz} \cI_{\cP_{l}}^{d}$. 

Next, for proving $\bigcup_{l \in \Nz} \cI_{\cP_{l}}^{d} \subseteq \cC$, we have to show that $\cI_{\cP_{l}}^{d} \subseteq \cC$ for some $l \in \Nz$. So take $g \in \cI_{\cP_{l}}^{d}$ for some $l \in \Nz$. To prove $g \in \cC$, it is sufficient to prove nonnegativity of $g(x)$ for $x \in \Delta^S$. Let us choose an arbitrary $x \in \Delta^S$, then $x \in \Delta^k$ for some $\Delta^k \in \cP_{l}$. By direct use of Lemma~\ref{simpconv}, we get $g(x)=G[x,x,\dots,x] \ge 0$ for all $x \in \Delta^S$. 

Lastly, since $\cC=\overline{\inn \cC}$, we get $\cC=\overline{\bigcup_{l \in \Nz} \cI_{P_{l}}^{d}}$.
\end{proof}

%See Algorithm~\ref{alg:Disc} for an implementation of Proposition \ref{prop:InnerApp}.

%Let $C=\left\{g \in \R[x] \vert\ g(x) \ge 0, \forall x \in \R_{+}^n \right\}$ the convex cone of copositive polynomials. For an arbitrary subset $T \subseteq \R_{+}^n$, let us define $Pos(T):= \left\{g \in \R[x] \vert\ g(x) \ge 0, \forall x \in T \right\}$. We can observe that any copositive polynomial $g$ satisfies $g \in Pos(T)$ so we have an outer approximation of $C$ in the form $C \subset Pos(T)$. Note that for some infinite subsets $T$, this approximation will be exact i.e. $Pos(T)=C$. For example, we can use the standard simplex $T=\Delta^n$. We denote $\mathbb{N}_{r}^n = \left\{m \in \mathbb{N}^n \vert\ \sum_{i=1}^n m_{i}=r\right\}$.\\
%This discretization method let us add a constraint on the state $x$ that will let us apply numerical construction of $V$ using convex optimization. In our examples in the next section, we will add the constraint $x \in B_{R}$ where $B_{R}=\left\{x \in \R_{+}^n \vert\ x_{1}^2+...+x_{n}^2 \le R^2\right\}$ with $R$ big enough.
%Then, the inequalities in \eqref{Algoineq} will be replaced by
%\begin{subequations}\label{foralgo}
%\begin{gather}
%V(x)=h(x) \ge 0, \enspace \forall x; \enspace x \in K, \enspace R^2-\sum_{i=1}^{n}x_{i}^2 \ge 0\\
%h_{0}(x) \ge 0, \enspace \forall x;\enspace x \in \inn K,\enspace R^2-\sum_{i=1}^{n}x_{i}^2 \ge 0 \\
%h_{i}(x) \ge 0, \enspace \forall x;\enspace x \in S_{i},\enspace R^2-\sum_{i=1}^{n}x_{i}^2 \ge 0
%\end{gather}
%\end{subequations}

\subsection{Sum-of-Squares (SOS) Method}\label{sec:sosMethod}
A commonly employed tool for checking the positivity of a polynomial  is to write it in the form of a sum of squares of other polynomials. While testing positivity is a computationally hard problem, the question of finding an SOS decomposition of a polynomial is actually a semidefinite program \cite{Powers98}.
The crux of such ideas can be found in \cite{Par00} and its application to copositivity is sketched in \cite{Parr00}. 

%In this section, we take the case $F=I$ the identity matrix i.e. $K=\R_{+}^n$. 

The basic idea is to get rid of the constraint $x \in \R_+^n$. We let $x_i=y_i^2 $, $i \in \left\{1,\cdots, n \right\}$ be the change of variable where $y^2$ is the short-hand for $(y_1^2, \dots, y_n^2)$. 
Clearly we have 
\[
h(x) \ge 0, \enspace \forall x \in \R_+^n \iff h(y^2) \ge 0, \enspace \forall y \in \R^n.
\]
Then, the inequalities \eqref{Algoineq} are rewritten as follows
\begin{subequations} \label{ParAlgoineq}
\begin{gather}
 P_{h}(y):=h(y^2) \ge 0, \enspace \forall y \in \R^n \\
 P_{s_{0}}(y):=s_{0}(y^2) \ge 0, \enspace y_i \ne 0, \forall i \\
 P_{s_{i}}(y):=s_{i}(y^2) \ge 0, \enspace y_i = 0, \enspace i \in \left\{1,...,n\right\}
\end{gather}
\end{subequations}
where $h, s_o, s_i$ are homogeneous polynomials.

Next, we define the polynomials
\begin{subequations} \label{polyd}
\begin{gather}
 P_{h}^{(d)}(y):=\|y\|^{2d}P_{h}(y)\\
 P_{s_{0}}^{(d)}(y):=\|y\|^{2d}P_{s_{0}}(y)\\
 P_{s_{i}}^{(d)}(y):=\|y\|^{2d}P_{s_{i}}(y)
\end{gather}
\end{subequations}
where $d$ is an integer. It is obvious that inequalities \eqref{ParAlgoineq} are satisfied if and only if
\begin{subequations} \label{dposit}
\begin{gather}
 P_{h}^{(d)}(y) \ge 0, \enspace \forall y \in \R^n\\
 P_{s_{0}}^{(d)}(y) \ge 0, \enspace y_i \ne 0, \forall i\\
 P_{s_{i}}^{(d)}(y) \ge 0, \enspace y_i = 0, \enspace i \in \left\{1,...,n\right\}.
 \end{gather}
\end{subequations}
%A sufficient condition that provides inequalities \eqref{dposit} is the existence of $d$, $d_{0}$ and $d_{i}$ for $i \in \left\{1,\dots,n\right\}$ such that $P_{h}^{(d)}(y)$, $P_{s_{0}}^{(d_{0})}(y)$ and $P_{s_{i}}^{(d_{i})}(y)$ can be written as sum-of-squares of polynomials.

\begin{prop}\label{prop:sospoly}
For the homogeneous copositive functions $h$, $s_0$ and $s_i$, $i \in \left\{1,\dots,n\right\}$, there exists $d \in \N$ sufficiently large such that the polynomials $P_{h}^{(d)}$, $P_{s_{0}}^{(d)}$ and $P_{s_{i}}^{(d)}$ are SOS. 
\end{prop}
%\begin{lem}
%Let $g(x)$ be a homogeneous copositive polynomial function where $x \in K$. There exists $l \in \Nz$ sufficiently large such that the polynomial $P_{g}^{(l)}$ has a sum of squares decomposition.
%\end{lem}

\begin{proof}
We will carry out the proof only for $h$ and it will be similar for the other polynomials.
%By assumption, $h$ is copositive which means $h(x) \ge  0$, $\forall x \in \R^n$ ; $Fx \ge 0$. To get rid of the positivity constraint, we impose $(Fx)_{i}=y_{i}^2$ where $y \in \R^n$. We define $P_{h}$ by $P_{h}(y)=h(y^2)$ and $P_{h}^{(l)}$ by $P_{h}^{(l)}(y)=\|y\|^{2l}P_{h}(y)=(\sum_{i=1}^n y_{i}^2)^l P_{h}(y)$ %where $l$ is any positive integer. 
Let 
\begin{subequations}\label{2sets} \nonumber
\begin{gather}
K_{n}^d:=\left\{h \in \R[x] \vert\ P_{h}^{(d)} \enspace \text{SOS} \right\} \\
C_{n}^d:=\left\{h \in \R[x] \vert\ P_{h}^{(d)} \enspace \text{has positive coefficients} \right\}.
\end{gather}
\end{subequations}
We notice that $C_{n}^d \subseteq K_{n}^d$ because if $P_{h}^{(d)}(y)$ has only positive coefficients then the polynomial $P_{h}(y)=h(y^2)$ is SOS and since $P_{h}(y)$ is multiplied by $\|y\|^{2d}$, it follows that $P_{h}^{(d)}(y)$ is SOS. So we just need to prove that $P_{h}^{(d)}(y)$ has positive coefficients.

The copositivity of $h$ is equivalent to the positivity of $P_{h}$. And since $h$ is homogeneous, this will be also equivalent to the positivity of $P_{h}$ on the unit ball which means $P_{h}(y) \ge 0$, $\forall y \in \R^n$, $\sum_{i=1}^n y_{i}^2=1$. By substituting $y_{i}^2$ by $z_{i}$, we obtain $P_{h}(z) \ge 0$ , $\forall z \ge 0$, $\sum_{i=1}^n z_{i}=1$.

Let us now recall P\'olya's Theorem, see \cite{Lau09} and \cite{PowRez01} for the proof.

\begin{thm} {\bf (P\'olya's Theorem)} Let $f \in \R[x]$ be homogeneous. If $f \ge 0$ on the simplex $\left\{x \ge 0 \vert\ \sum_{i=1}^n x_{i}=1 \right\}$, then there exists a sufficiently large $l \in \Nz$ for which the polynomial $(\sum_{i=1}^n x_{i})^l f(x)$ has all its coefficients nonnegative.
\end{thm}

Applying this Theorem to the homogeneous polynomial $P_{h}(z)$, we obtain that for sufficiently large $d \in \Nz$, all the coefficients of the polynomial $P_{h}^{(d)}(z)=(\sum_{i=1}^n z_{i})^d P_{h}(z)$ are positive. Then $P_{h}^{(d)}$ is SOS in view of the fact that $C_{n}^d \subseteq K_{n}^d$.
\end{proof}

To sum up this section, the foregoing result allows us to write an algorithm to compute the polynomials $P_{h}^{(d)}$, $P_{s_{0}}^{(d)}$ and $P_{s_{i}}^{(d)}$ in the form of SOS,
the result is then used to get a homogeneous copositive Lyapunov function,
see Algorithm~\ref{alg:SOS} in Appendix~\ref{app:algoSOS}.

% Algorithm~\ref{alg:SOS} provides the pseudocode for performing this numerical computation. We use a
% command in YALMIP which is solvesos that takes SOS polynomials and return the function that we search for.

%% file: examples.tex
% !TEX root = sosCertEVIs.tex
%%%%%%%%%%%%%%%%%%%%%%%%%%%%%%%%%%%

\section{Examples and Simulations}\label{sec:simul}
In this section, we compute copositive polynomial Lyapunov functions for complementarity systems by implementing our two methods (discretization and SOS).  
In our examples, the YALMIP toolbox in Matlab is used to input the LP and SOS optimization problems and solve them with the conic solver MOSEK.

%\begin{exam}(Quadratic polynomial Lyapunov function). Consider the dynamical system
%\begin{equation}
%\dot x=Ax+\eta
%\end{equation}
%with $x \in \R_{+}^2$ and $A = \begin{bmatrix} -1 & -2\\ -1 &-1\end{bmatrix}$. By using inequalities \eqref{Algoineq} and replacing the nonnegativity by SOS, we can search for a quadratic polynomial Lyapunov function $V$ by solving an SOS program. We fix the degree 2 of $V$ and the SOS program will give us:
%\begin{equation}
%V(x)=0.1288x_{1}^2+0.1154x_{1}x_{2}+0.1676x_{2}^2
%\end{equation}
%We can check that this polynomial is positive definite by checking that the eigenvalues of the corresponding matrix $P = \begin{bmatrix} 0.1288 & \frac{0.1154}{2} \\ \frac{0.1154}{2}  & 0.1676\end{bmatrix}$ are strictly positive.
%\begin{figure}[hbtp]
%\centering
%includegraphics[width=8cm]{Levelset.jpg}
%\caption{Level set of the Lyapunov function $V$ in Example 4.}
%\end{figure}

%\end{exam}

\begin{exam}[Quadratic Lyapunov function by the discretization method]\label{ex:discret}
Consider system \eqref{eq:compSys} with $f(x)=Ax$ and $A = \begin{bsmallmatrix} -1 & -2\\ -1 &-1\end{bsmallmatrix}$ and $K=\R_+^2$. We apply the discretization method on the standard simplex $\Delta^S := \left\{x \in \R_+^2 \vert\ \|x\|_{1}=1\right\}$ of Algorithm~\ref{alg:Disc}. Starting with a degree 2 polynomial, we solve for its coefficients at the vertices of $\Delta^S$. This procedure in Algorithm~\ref{alg:Disc} is applied by partitioning all the simplices at each step by a factor of half and solving certain inequalities at the vertices of resulting simplices. For this example, we found 
$
V(x)=x_{1}^2+x_{1}x_{2}+x_{2}^2
$
in four iterations, that is, Algorithm~\ref{alg:Disc} terminates with $\delta = 1/16$.
\end{exam}

\begin{exam}[Cubic Lyapunov function by the discretization method]\label{ex:discret1}
Consider system \eqref{eq:compSys} with $K=\R_+^2$ and
\begin{equation}
f(x) = \left[\begin{array}{c}
-x_1^2 -2x_2^2 +x_1x_2\\
-x_1^2 -x_2^2 +2x_1x_2\\
\end{array}\right].
\end{equation}
Applying the discretization method of Algorithm~\ref{alg:Disc}, after 4 iterations with $\delta=\frac{1}{16}$, we obtain 
\begin{equation}
V(x)=x_{1}^3 + \frac{3}{2} x_1x_2^2 + \frac{3}{2} x_2x_1^2 + \frac{1}{2} x_2^3 .
\end{equation}
\end{exam}

\begin{exam}[Quadratic Lyapunov function by SOS method]
Consider system \eqref{eq:compSys} with $f(x)=Ax$ and $A = \begin{bsmallmatrix} -1 & 10\\ 0 &-2\end{bsmallmatrix}$ and $K=\R_+^2$.
Following Algorithm~\ref{alg:SOS}, we express positivity condition on desired polynomials by requiring them to be SOS, and use the YALMIP command {\tt solvesos} which calls a semidefinite solver to yield the desired coefficients. We obtain
\begin{equation}
V(x)=0.1x_{1}^2+0.1916x_{1}x_{2}+1.1137x_{2}^2.
\end{equation}
\end{exam}

\begin{exam}[A copositive quadratic Lyapunov function that is not positive definite] 
\label{ex:3d}
\begin{figure}[!t] 
\centering
\includegraphics[height=5cm]{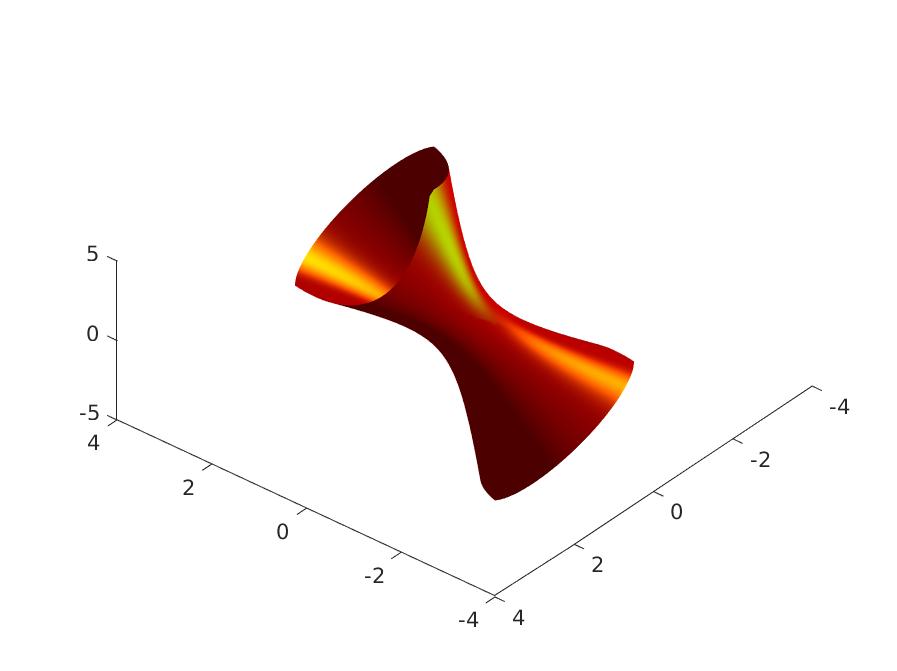}
\caption{Non-convex unit level set of the quadratic Lyapunov function $V$ in Example~\ref{ex:3d}.}
\end{figure}

Consider system \eqref{eq:compSys} with $f(x)=Ax$ and $A = \begin{bsmallmatrix} -1 & -3 & -2\\ -5 & 1 & -1 \\ 3 & -10 & -2\end{bsmallmatrix}$ and $K=\R_+^3$.
Applying the SOS method of Algorithm~\ref{alg:SOS} we obtain
\begin{equation*}
V(x)= 2.3234x_{1}^2+3.6729x_{1}x_{2}+1.7352x_{2}^2+1.1273x_{1}x_{3} +2.6769x_{2}x_{3}+1.2820x_{3}^2.
\end{equation*}
This polynomial is not positive definite since one of the eigenvalues of the corresponding matrix is negative. The unit level set of this polynomial Lyapunov function is shown in Figure 2.

\end{exam}

These examples just provide an illustration of two classes of algorithms primarily used for checking positivity or copositivity of polynomials, and how they can be used for computing Lyapunov functions with constrained dynamics. The survey article \cite{BomDurTeo12} provides an overview of these methods, along with some other techniques, which appear in general in the literature on checking copositivity. Further questions such as using other algorithms or comparing computational complexity of different methods require further investigation.

%% file: appendixA.tex
% !TEX root = sosCertEVIs.tex
%%%%%%%%%%%%%%%%%%%%%%%%%%%%%%%%%%%
\appendices
\section{Results from Complementarity Theory}\label{app:basicsComp}
The reference book for this topic is \cite{cottle1992}, see also \cite{FaccPang03}.
\begin{defn}[Complementarity Problem]\label{def:compProb}
Let $F:\R^{n} \rightarrow \R^{n}$ and $\eta \in \R^{n}$. The problem $\eta \geq 0$, $F(\eta)\geq 0$, $\eta^{\top}F(\eta)=0$ is a complementarity problem (CP) with unknown $\eta$, written compactly as $0 \leq \eta \perp F(\eta) \geq 0$. When $F(\eta)=M\eta+q$ for a matrix $M$ and a vector $q$, this is a linear complementarity problem denoted $\lcp(q,M)$. Similarly, for a given closed, convex cone $K \subset \R^n$, the problem of finding $\eta \in K^\star$ such that $K^\star \ni \eta \perp F(\eta) \in K$ is termed as the cone-complementarity problem, and for $F(\eta)=M\eta+q$, it is a linear cone complementarity problem, denoted $\lccp(q, M, K)$.

Based on discussions in \cite[Chapter 2]{FaccPang03}, complementarity problems can be reformulated as nonlinear equations, or optimization problems. For a closed convex cone $K$, one can also write the solution of $\lccp(q,M,K)$ as the solution to the following optimization problem:
\begin{gather} \label{optimprob}
\begin{aligned}
\min_{\eta \in K^\star} & \ \eta^\top (q + M\eta)\\
\text{s.t.} & \ M\eta + q \in K.
\end{aligned}
\end{gather}
\end{defn}

\begin{defn}\label{definitionB2}
%A matrix $M \in \R^{n \times n}$ is positive (semi) definite if for all $x \in \R^{n}$ one has $x^{T}Mx>0\;(\geq 0)$ for all $x \neq 0$. It is denoted $M \succ 0\;(\succcurlyeq 0)$. It is not necessarily symmetric. 
A matrix $M \in \R^{n \times n}$ is a P-matrix if all its principal subdeterminants (or principal minors) are positive. It is a P$_{0}$-matrix if its principal minors are non negative.
%It is a copositive matrix on the set $K$  if $x^{T}Mx \geq 0$ for all $x \in K$. 
\end{defn}

\noindent
We have $M \succ 0 \Rightarrow M$ is a P-matrix, $M \succcurlyeq 0 \Rightarrow M$ is a P$_{0}$-matrix and a copositive matrix on $\R^{n}_{+}$. One usually considers copositivity over convex sets \cite{hiriart2010}. Yet even in this case copositivity is hard to characterize. Many more matrix classes which are useful in complementarity theory exist \cite{cottle1992}. The following result is central in complementarity theory. 
\begin{thm}
The $\lcp(q,M)$ has a unique solution for any $q \in \R^n$ if $M\in \R^{n \times n}$ is a P-matrix. 
\label{theoremLCP}
\end{thm}

For the analysis carried out in this paper, it is important to know how the solution of an $\lcp$, or $\lccp$ in general, changes if we modify one of the parameters.

\begin{prop}\label{prop:solScaleLcp}
Given a closed convex cone $K$ and a P-matrix $M$, let $\eta$ denote the solution of $\lccp(q,M,K)$ and $\eta_\alpha$ denote the solution of $\lccp(\alpha q, M,K)$, for some $\alpha > 0$. Then, it holds that $\eta_\alpha =\alpha \eta$.
\end{prop}

\begin{proof}
Let  $\eta \in \lccp(q,M, K)$. Clearly, for each $\alpha > 0$,
\begin{align*}
\eta\in K^\star & \Leftrightarrow \ \alpha \eta \in  K^\star\\
M\eta+q \in K & \Leftrightarrow \  \alpha (M \eta +q) = M(\alpha\eta) + (\alpha q)\in K \\
\eta^\top (q + M \eta) =0 & \Leftrightarrow \  (\alpha \eta)^\top (M(\alpha\eta) + (\alpha q)) = 0.
\end{align*}
and hence $\alpha \eta \in \lccp(\alpha q, M,K)$. Since the solution to such an $\lccp$ are unique, it follows that $\eta_\alpha = \alpha \eta$.
\end{proof}

The next statement concerns also the sensitivity of the solution of an $\lccp$ with respect to one of its parameters. The results given in \cite[Section~2]{MangShia87} and \cite{Pang87} focus on Lipschitz continuity of the solution to $\lcp$ problems, and they can be modified to get the following statement:
\begin{prop}\label{prop:boundEta}
Consider system \eqref{eq:compSys} under Assumption \ref{lip}. Let $(x,\eta): [0,\infty) \to \R^{2n}$ denote the solution with an admissible initial condition $x(0) \in K$. Then, there exists a constant $C > 0$ such that for each $t \ge 0$,
\begin{equation}\label{eq:boundEta}
\| \eta(t) \| \le C \| f(x(t)) \| .
\end{equation}
\end{prop}

\section{Proof of Lemma~\ref{lem:lipCont}}\label{app:globLip}
For $f$ locally Lipschitz in \eqref{eq:compSys}, there exists a continuous positive definite function $\beta:\R^n \to \R_{+}$, such that $\beta(x)f(x)$ is globally Lipschitz on $\R^n \setminus \{0\}$, \cite[Lemma~4.10]{ClarLedy98}. Set $\wh f(x) := \beta(x)f(x)$ in \eqref{eq:compSysLip}. We first prove the second item: if $\wh V$ is a continuously differentiable Lyapunov function for \eqref{eq:compSysLip}, then there exists a class $\cK$ function $\wh \gamma$ such that $\langle \nabla \wh V, \beta(x)f(x)\rangle \le -\wh \gamma(\| x \|)$ for $x \in \inn(K)$ and $\langle \nabla \wh V, \beta(x)f(x) + \eta \rangle \le - \wh \gamma(\| x \|)$ for $x \in \bd(K)$. By choosing a class $\cK$ function $\gamma$ such that $\gamma(\| x \|) < \frac{1}{\beta(x)}\wh \gamma(\| x \|)$, and using Proposition~\ref{prop:solScaleLcp}, it follows that $V = \wh V$ is a continuously differentiable Lyapunov function for \eqref{eq:compSys}.

To prove the first item, we need to show that the origin of \eqref{eq:compSysLip} is globally exponentially stable. Let $z \in [0,\infty)$ be a solution to \eqref{eq:compSysLip}, and let $\rho(t) = \int_0^t \beta(z(s)) \, ds$. Using Proposition~\ref{prop:solScaleLcp} and the chain rule for differentiation, it follows that $x(t) = z(\rho^{-1}(t))$ is a solution of \eqref{eq:compSys}. Thus, for every solution $z$ of \eqref{eq:compSysLip}, there exists a solution $x$ of \eqref{eq:compSys} such that $z(t) = x(\rho(t))$. Lyapunov stability of the origin of \eqref{eq:compSysLip} thus follows by inspection. Suppose that there exists a solution $\overline z$ such that $\overline z(t)$ does not converge to the origin as $t \to \infty$, then $\lim_{t\to \infty}\rho(t) = +\infty$. Let $\overline x$ be a solution to $\eqref{eq:compSys}$ such that $\overline z(t) = \overline x (\rho(t))$ and since \eqref{eq:compSys} is asymptotically stable, we have $\lim_{t\to\infty} \overline x (\rho(t)) = 0$, which is a contradiction. Hence, $\overline z$ converges to the origin as well.

\section{Homogeneous Lyapunov Function}\label{app:proofHomo}

\begin{proof}[\textbf{Proof of Proposition~\ref{prop:homogLyapfct}}]
The key ingredient required for applying the construction of \cite{Ros92} is to show that if $f$ is homogenous of degree $d \ge 1$, then
\[
\lccp(f(\lambda x), I, \cT_K(\lambda x)) = \lambda^{d} \lccp(f(x), I, \cT_K(x)),
\]
that is the nonsmooth multiplier $\eta$ respects the same homogeneity as the function $f(\cdot)$. This indeed follows from Proposition~\ref{prop:solScaleLcp} given in Appendix~\ref{app:basicsComp}.

The function $\overline{W}$ is well defined since we have $W(x) \rightarrow +\infty$ as $\| x \| \rightarrow +\infty$ and vanishes at $0$. Besides, we can find two numbers $\underline a > 0$ and $\overline a > 0$ such that
$W(\lambda x) \le 1$, for $\| x \| \in \left[0.5,2\right]$, $\lambda \le \underline a$, and $W(\lambda x) \ge 2$, for $\| x \| \in \left[0.5,2\right]$, $\lambda \ge \overline a$.
Then, for all $x \in \R^n$ satisfying $\| x \| \in \left[0.5,2\right]$, we have
\[
\overline W(x)= \int_{0}^{\infty} \frac{1}{\lambda^{k+1}} (a \circ W)(\lambda x)  \, d\lambda + \frac{1}{k \overline a^k}.
\]
It is obvious that $\overline {W}$ is $\cC^1$ on the set $\left\{x \vert\ \| x \| \in \left[\frac{1}{2},2\right]\right\}$. So we have
\[ 
\frac{\partial\overline{W}}{\partial x_i}(x)= \int_{0}^{\infty} \frac{\lambda}{\lambda^{k+1}} \nabla a(W(\lambda x)).\frac{\partial W}{\partial x_i}(\lambda x) \, d\lambda.
\]
By the homogeneity of $f$ and since $\eta$ satisfies $\eta_{\lambda x}=\lambda^{d}\eta_x$, we obtain
\begin{equation} 
\left\langle \nabla \overline W(x), f(x) + \eta_x \right\rangle = \int_{0}^{\infty} \frac{1}{\lambda^{d+k+1}} \nabla a(W(\lambda x)) \left \langle \nabla W(\lambda x), f(\lambda x) + \eta_{\lambda x} \right\rangle \, d\lambda.
\end{equation}
Since $\nabla a(s) > 0$ for some $s \in (1,2)$ and $W$ is a Lyapunov function then, for $\frac{1}{2} < \| x \| < 2$, the right-hand side is negative.

Homogeneity of $\overline W$ follows by a change of variable of integration. Therefore, we get $\overline W$ is $\cC^1$ on $\R^n\backslash\lbrace{0}\rbrace$ and cone-copositive Lyapunov function with respect to $K$.
\end{proof}

\section{Discretization algorithm}\label{app:algoDisc}

A pseudocode which allows us to compute Lyapunov function based on discretization of simplices is given in Algorithm~\ref{alg:Disc}.

\begin{algorithm}[!h]\caption{Discretization method in $\R_+^n$}\label{alg:Disc}
\small
\textbf{Input:} vector field $f$, maximum degree $d_{\max}$ (resp. $r_{\max}$) of the numerator (resp. denominator) of Lyapunov function, minimum diameter of the simplical partition $\epsilon$.\\
\textbf{Output:} {either a copositive Lyapunov function $V$, or an error message.} \\
%\item $h(x) \longleftarrow H[\underbrace{x,x,...,x}_{\text{of cardinal d}}]$\\
%\item $p(x)\longleftarrow -\left\langle \nabla h(x),f(x)+\eta \right\rangle$ \\
%\item $H$, $P$ $\longleftarrow$  the tensors corresponding to $h$, $p$ \\
%$\delta(\cP) \longleftarrow \max_{k \in \left\{1,\dots,m\right\}} \max_{i,j \in \left\{1,\dots,n\right\}} \|v_{i}^k - v_{j}^k\|$ \\
%$\{\cP_{l}\}$ $\longleftarrow$ sequence of simplicial partitions of $\Delta^S$\\
$\Delta^S \longleftarrow \{x \in \R_+^n \, \vert \,  \|x\|_{1}=1\}$\\
$\delta$ $\longleftarrow$ $1$ \\
\While {$\delta > \epsilon$}{ 
$Q^\ell$ $\longleftarrow$ {vertices of simplex $\Delta^\ell$ of a simplicial partition $\{\Delta^1,\dots,\Delta^m \}$ of $\Delta^S$ with diameter $\delta$} \\
%$E_{\cP}$ $\longleftarrow$ {the edges of simplices in $\cP$} \\
%$p^\ell \longleftarrow \vert Q^\ell\vert$ \\
%$v_{1}^k,\dots,v_{n}^k$ $\longleftarrow$ vertices of simplex $\Delta^k$, $k \in {1,\dots,m}$ \\
\ForAll {$r=0,1,2,\dots,r_{max}$}{
\ForAll {$d=1,2,\dots,d_{max}$}{
\ForAll {$\ell=1,2,\dots,m$}{
$h$ $\longleftarrow$ homogeneous polynomial of degree $d$ and $n$ variables with unknown coefficients\\
%\begin{enumerate}
\ForAll {$i=1,2,\dots, \vert Q^\ell \vert$}{
	$x_i$ $\longleftarrow$ $v_i \in Q^\ell$\\
	$\eta_{x_i}$ $\longleftarrow$ $\lccp(f(x_i), I, \cT_{\R^n_+}(x_i))$ \\
	$s_k(x_i)$ $\leftarrow$ $- \| x_i \|_{2}^2 \left\langle \nabla h(x_i), f(x_i)+ \eta_{x_i} \right\rangle +$ \\
	$2rh(x_i) \left\langle x_i, f(x_i)+ \eta_{x_i} \right\rangle$, $k = 0,\dots,n$
	}
\ForAll {$j = 1,2,\dots, {\vert Q^\ell\vert \choose d} $}{	
$Q^\ell_j$ $\longleftarrow$ $j^{th}$ combination of $d$ vertices in $Q^\ell$ \\
Solve the LP problem in the coefficients of $h$ corresponding to the constraints $H[q_1,\dots,q_d] \ge 0$
and $S_k[q_1,\dots,q_d] \ge 0$ where $H,S_k$ denote the 
tensors of $h,s_k$ and $\left\{q_1,\dots,q_d\right\} \in Q^\ell_j$, $k = 0, \dots, n$ \\
\If {the LP problem is feasible} {
return $V(x) = \frac{h(x)}{\| x\|_2^{2r}}$ \\
}
%\Else{
%{\em flag} $\longleftarrow$ {\tt false}
%}
}
%\end{enumerate}
}
}
}
$\delta \longleftarrow \frac{\delta}{2}$
}
%\uIf {flag == {\tt true}} {
%return $V(x) = \frac{h(x)}{\| x\|_2^{2r}}$
%}
%\Else{%{flag == {\tt true}} {
display(``Lyapunov function not found'')
%}
\end{algorithm}

%\fi

\section{Sum-of-Squares Algorithm}\label{app:algoSOS}

The pseudocode based on SOS decomposition is given in Algorithm~\ref{alg:SOS} given below. In addition to the procedure outlined in Section~\ref{sec:sosMethod}, we use the YALMIP command {\tt solvesos} to model and solve the SOS optimization problem: It computes the unknown coefficients $h_i$ that we associate with the polynomial $h \in \R^q[x]$, while minimizing $\sum h_i^2$, under the constraint that $P_h^{(d)}(x)$, $P_{s_k}^{(d)}(x)$, $k = 0, \dots, n$ must be SOS for some $d \in \Nz$.

\input{algoNumeric}

%% file: algoNumeric.tex
% !TEX root = sosCertEVIs.tex
%%%%%%%%%%%%%%%%%%%%%%%%%%%%%%%%%%%

\begin{algorithm}[!h]\caption{SOS Approximations of Lyapunov Functions}\label{alg:SOS}
\small
\textbf{Input:} vector field $f$, maximum degree $q_{\max}$ (resp. $r_{\max}$) of the numerator (resp. denominator) of Lyapunov function, maximum degree $d_{\max}$  for expressing homogeneous polynomials\\
\textbf{Output:} either a copositive Lyapunov function $V$, or an error message. \\
%$m$ $\longleftarrow$ number of faces of $K$ \\
\ForAll{$r=1,2,...,r_{\max}$}{
\ForAll{$q=1,2,...,q_{\max}$}{
   1. \enspace $h$ $\longleftarrow$ homogeneous polynomial of degree $q$ and $n$ variables with unknown coefficients $h_i$\\
   \hspace{0.4cm} $s_0(x)=-\|x\|_{2}^2 \left\langle \nabla h(x), f(x) \right\rangle +2rh(x) \left\langle x, f(x) \right\rangle$\\ %\new{\bf Why do we say that $r_{\max}$ is the maximum degree of the denominator ?}\\
   \ForAll{$k=1,2,...,n$}{
   $\eta_{k}(x)$  $\longleftarrow$ $\lccp(f(x), I, \cT_{\R^n_+}(x))$, for $x \in S_k$\\
   $s_k(x) = -\|x\|_{2}^2 \left\langle \nabla h(x), f(x)+ \eta_k \right\rangle $ \\
   \hspace{2cm} $ +2rh(x) \left\langle x, f(x)+ \eta_k \right\rangle$\\
   }
   2.  $P_{h}(y) \longleftarrow h(y^2)$\\
      \enspace \enspace $P_{s_k}(y) \longleftarrow s_{k}(y^2)$, $k = \{0, \dots, n\}$\\
   3. \ForAll {$d=0,...,d_{\max}$}{
     $P_{h}^{(d)}(y) \longleftarrow \|y\|^{2d}P_{h}(y)$\\
     $P_{s_k}^{(d)}(y)\longleftarrow \|y\|^{2d}P_{s_{k}}(y)$, $k = \{0, \dots, n\}$\\
     {\sc solvesos} $\Big(\textrm{\sc sos}(P_{h}^{(d)}), \textrm{\sc sos}(P_{s_0}^{(d)}), \textrm{\sc sos}(P_{s_k}^{(d)}), $\\ \hspace{1.5cm} $ \sum_{i}h_i^2, [~], [h_i] \Big)$
     
   4. \If {the SOS program is feasible} {
       return $V(x) = \frac{h(x)}{\| x\|_2^{2r}}$
   }
   }
}
}
display(``Lyapunov function not found'')
\end{algorithm}